\documentclass[a4paper,oneside]{article}
\pdfoutput=1
\RequirePackage{amsmath,amsthm,amssymb,mathtools}
\RequirePackage{natbib}
\RequirePackage{pgfplots}
\pgfplotsset{compat=1.17}
\RequirePackage{algorithm2e}
\RequirePackage[margin=1.2in]{geometry}
\RequirePackage{float}
\RequirePackage{libertine}
\RequirePackage{courier}
\RequirePackage{mathpazo}
\RequirePackage{makecell}
\RequirePackage{xcolor}
\RequirePackage{pifont}
\RequirePackage{authblk}
\RequirePackage[backref=page]{hyperref}
\RequirePackage{enumerate}
\RequirePackage{titlesec}
\RequirePackage{mathrsfs}

\definecolor{darkgreen}{RGB}{0,150,0} 

\definecolor{darkbrown}{RGB}{170,100,0} 
\definecolor{darkdarkbrown}{RGB}{110,70,0} 
\titleformat{\section}
{\color{darkdarkbrown}\normalfont\Large\bfseries}
{\color{darkdarkbrown}\thesection}{1em}{}
\titleformat{\subsection}
{\color{darkdarkbrown}\normalfont\Large\bfseries}
{\color{darkdarkbrown}\thesubsection}{1em}{}

\hypersetup{
	colorlinks=true,
  linkcolor=darkbrown,
  citecolor=darkbrown,
  filecolor=darkbrown,
  urlcolor=darkbrown
}

 \renewcommand*{\backrefalt}[4]{%
    \ifcase #1%
     \or (page:~#2)%
     \else (pages:~#2)%
    \fi%
    }

    \makeatletter
\def\@fnsymbol#1{\ensuremath{\ifcase#1\or \dagger\or \ddagger\or
   \mathsection\or \mathparagraph\or \|\or **\or \dagger\dagger
   \or \ddagger\ddagger \else\@ctrerr\fi}}
    \makeatother

\newtheorem{theorem}{Theorem}[section]
\newtheorem{lemma}{Lemma}[section]
\newtheorem{definition}{Definition}[section]
\newtheorem{proposition}{Proposition}[section]
\newtheorem{example}{Example}[section]
\newtheorem{corollary}{Corollary}[section]
\newtheorem{remark}{Remark}[section]
\newtheorem{oproblem}{Open problem}[section]

\makeatletter
\renewcommand\paragraph{%
  \@startsection{paragraph}
    {4}
    {\z@}
    {3.25ex \@plus1ex \@minus.2ex}
    {-1em}
    {\normalfont\normalsize\bfseries\maybe@addperiod}%
}
\newcommand{\maybe@addperiod}[1]{%
  #1\@addpunct{.}%
}
\makeatother

\providecommand{\keywords}[1]
{
  \small	
  \textbf{\textit{Keywords---}} #1
}
\DeclarePairedDelimiter{\ceil}{\lceil}{\rceil}
\DeclarePairedDelimiter{\floor}{\lfloor}{\rfloor}

\DeclareMathOperator*{\argmin}{arg\,min}

\DeclareMathOperator*{\diag}{diag}

\newcommand{\R}{\mathbb{R}}

\newcommand{\N}{\mathbb{N}}

\newcommand{\nrm}[1]{\left\Vert #1 \right\Vert}

\newcommand{\ps}{_{\mathsf{ps}}}
\newcommand{\dps}{_{\mathsf{dps}}}
\newcommand{\dpsK}{_{\mathsf{dps},[K]}}
\newcommand{\dpshatK}{_{\mathsf{dps},[\widehat{K}]}}
\newcommand{\psK}{_{\mathsf{ps},[K]}}
\newcommand{\pshatK}{_{\mathsf{ps},[\widehat{K}]}}

\newcommand{\psKeps}{_{\mathsf{ps},[K_{\eps, \gamma \ps}]}}

\newcommand{\eqdef}{\doteq}
\newcommand{\trn}{^\intercal}
\newcommand{\eps}{\varepsilon}
\newcommand{\abs}[1]{\left| #1 \right|}
\newcommand{\set}[1]{\left\{ #1 \right\}}
\newcommand{\PR}[2][]{\mathbb{P}_{#1}\left( #2 \right)}
\newcommand{\E}[2][]{\mathbb{E}_{#1}\left[ #2 \right]}
\newcommand{\calX}{\mathcal{X}}
\newcommand{\calE}{\mathcal{E}}
\newcommand{\calW}{\mathcal{W}}

\newcommand{\calP}{\mathcal{P}}
\newcommand{\innerpi}[2]{\langle #1, #2 \rangle_{\pi}}
\newcommand{\pimin}{\pi_\star}

\newcommand{\bigO}{\mathcal{O}}
\newcommand{\rev}{_{\mathsf{rev}}}
\newcommand{\Dpi}{D_\pi}
\newcommand{\tmix}{t_{\mathsf{mix}}}
\newcommand{\tv}[1]{\nrm{#1}_{\mathsf{TV}}}

\newcommand{\pred}[1]{\delta \left[#1\right]}
\newcommand{\Nmax}{N_{\max}}
\newcommand{\Nmin}{N_{\min}}

\newcommand{\EE}[1]{\mathbf{\mathcal{E}}_{#1}}
\newcommand{\estpssg}{\widehat{\gamma}_{\mathsf{ps}}}

\usepackage{fancyvrb} %

\title{Improved Estimation of Relaxation Time in \\  Non-reversible Markov Chains}

\author[1]{Geoffrey Wolfer \thanks{email: geoffrey.wolfer@riken.jp \\
The author is supported by the Special Postdoctoral Researcher Program (SPDR) of RIKEN.}}
\author[2]{Aryeh Kontorovich \thanks{email: karyeh@cs.bgu.ac.il}}
\affil[1]{RIKEN Center for AI Project}
\affil[2]{Department of Computer Science \protect\\ Ben-Gurion University of the Negev}

\date{\today}

\begin{document}

\maketitle

\begin{abstract}
We show that the minimax sample complexity for estimating the pseudo-spectral gap $\gamma_{\mathsf{ps}}$ of an ergodic Markov chain in constant multiplicative error is of the order of $$\tilde{\Theta}\left( \frac{1}{\gamma_{\mathsf{ps}} \pi_\star} \right),$$ where $\pi_\star$ is the minimum stationary probability, recovering the known bound in the reversible setting for estimating the absolute spectral gap \citep{hsu2019}, and resolving an open problem of \citet{pmlr-v99-wolfer19a}. Furthermore, we strengthen the known empirical procedure by making it fully-adaptive to the data, thinning the confidence intervals and reducing the computational complexity. Along the way, we derive new properties of the pseudo-spectral gap and introduce the notion of a reversible dilation of a stochastic matrix.
\end{abstract}


\tableofcontents

\section{Introduction}
\label{section:introduction}
Let $\{X_t\}_{t \in \N}$ be a time-homogeneous ergodic Markov chain over a state space $\calX$, with transition matrix $P$ and stationary distribution $\pi$. When $P$ is reversible, the convergence of the chain to its stationary distribution is roughly governed by the relaxation time, defined as the inverse of the absolute spectral gap $\gamma_\star$ of $P$, see  \eqref{definition:absolute-spectral-gap} and \eqref{eq:absolute-spectral-gap-controls-tmix}. When $P$ is non-reversible, the definition of the relaxation time can be extended to the inverse of the pseudo-spectral gap $\gamma \ps$, introduced by \citet{paulin2015concentration} and defined in \eqref{eq:pseudo-spectral-gap}.
Furthermore, $\gamma_\star$ and $\gamma\ps$ are spectral parameters of intrinsic interest, as they were shown to control rates of Bernstein type concentration inequalities \citep[Theorems~3.9~\&~3.11]{paulin2015concentration}, and yield bounds on the asymptotic variance in terms of the stationary variance \citep[Theorem~3.5~\&~3.7]{paulin2015concentration}. 
Under reversibility, the sample complexity of estimating $\gamma_\star$ to multiplicative error $\eps$ from a single trajectory of observations is known to be $\tilde{\Theta}(1/(\pimin \gamma_\star \eps^2))$ \citep{hsu2019,pmlr-v99-wolfer19a}, where $\pimin$ is the minimum stationary probability.
The non-reversible setting, first investigated by \citet{pmlr-v99-wolfer19a}, remains less understood, with
a large gap between known upper and lower bounds (see Table~\ref{table:comparison-state-of-the-art}).
In this paper, we set out the task of closing this chasm, and will mostly succeed doing so by recovering an upper bound on the rate of $\bigO(1/(\pimin \gamma \ps))$.

\subsection{Motivation and applications} 
Although a wide class of chains is known to be reversible (such as random walks 
on graphs or birth and death processes),
this condition is sufficiently restrictive as to exclude some natural and important applications.
One such
instance involves MCMC diagnostics for non-reversible chains, 
which have recently gained interest through {acceleration} methods. 
Indeed, while chains generated by the classical Metropolis-Hastings are reversible,
which is instrumental in
analyzing the stationary distribution, non-reversible chains may enjoy better {mixing properties} as well as improved asymptotic variance by reduction of backtracking behavior.
For theoretical and experimental results in this direction, see \citet{hildebrand1997rates, chen1999lifting, diaconis2000analysis, neal2004improving, sun2010improving, suwa2010markov, turitsyn2011irreversible,  chen2013accelerating, vucelja2016lifting, bierkens2016non, power2019accelerated, herschlag2020non, syed2022non}.
Another application is in reinforcement learning, where bounds on the mixing parameters of the underlying Markov decision process are routinely assumed  \citep{ortner2020regret, zweig2020provably, li2023accelerated}. Last but not least, many results from statistical learning or empirical process theory have been extended to Markov dependent data \citet{yu1994rates, mohri2007stability, steinwart2009learning, shalizi2013predictive, garnier2021machine, truong2022generalization, wolfer2021, kotsalis2022tractable, truong2022kernel, garnier2022hold}, and empirical estimates of the mixing parameters of the chain yield corresponding data-dependent generalization bounds.

\subsection{Main contributions}
We now give an informal overview of our contributions.
From a single trajectory of observations of length $m$, sampled according to an unknown ergodic transition matrix $P$ with minimum stationary probability $\pi_\star$ and pseudo-spectral gap $\gamma \ps$, and started from an arbitrary and unknown state, we obtain the following results.
\begin{enumerate}
    \item[$\star$] Theorem~\ref{theorem:pseudo-spectral-gap-estimation-absolute}. We upper bound the sample complexity of estimating $\gamma \ps$, to arbitrary additive error $\eps$, by 
    $$\tilde{\bigO} \left( \frac{1}{\eps^2 \pi _\star \gamma \ps}\right).$$
    \item[$\star$] Theorem~\ref{theorem:pseudo-spectral-gap-estimation-relative}. For estimating $\gamma \ps$ to constant multiplicative error, we obtain an upper bound of 
    $$\tilde{\bigO} \left( \frac{1}{\pi _\star \gamma \ps}\right).$$ 
    This bound is significantly stronger to the one in \citet{pmlr-v99-wolfer19a},
    which involved somewhat unnatural quantities (for instance a measure of how far a stochastic matrix is from being doubly stochastic). Furthermore, the new bound reduces to the known one in the reversible setting \citep{hsu2019}, for which they exist corresponding lower bounds, making it generally unimprovable and closing the open question of \citet[Remark~5]{pmlr-v99-wolfer19a}.
    See Section~\ref{section:related-work} and Table~\ref{table:comparison-state-of-the-art} for further comparison with the state-of-the-art.
    \item[$\star$] Theorem~\ref{theorem:pseudo-spectral-gap-estimation-relative-arbitrary}. For estimation of $\gamma \ps$ to arbitrarily small multiplicative error $\eps$, we recover an upper bound of
    $$\tilde{\bigO} \left( \frac{1}{\eps^2 \pi _\star \gamma^3 \ps}\right).$$ 
    \item[$\star$] Definition~\ref{definition:reversible-dilation}. We introduce the reversible dilation of a Markov chain, whose spectral properties are closely related to that of \citeauthor{fill1991eigenvalue}'s multiplicative reversiblization, and that is less expensive to compute when the stationary distribution is known.
    \item[$\star$] Theorem~\ref{theorem:confidence-intervals}. We improve the width of the confidence intervals of \citet{pmlr-v99-wolfer19a} in non-trivial regimes.
Furthermore, we make the procedure entirely adaptive to the data, in contrast to the previous intervals, which hard-coded an approximation error.
\end{enumerate}

\subsection{Outline}
Section~\ref{section:introduction} compares our results with the related work, lists our contributions, introduces notation, and presents some background on ergodic Markov chains and associated Hilbert space concepts.
In Section~\ref{section:minimax-estimation}, we exhibit a collection of features possessed by \citeauthor{paulin2015concentration}'s pseudo-spectral gap $\gamma \ps$, defined in \eqref{eq:pseudo-spectral-gap}.
We then proceed to state and prove our main results (Theorems~\ref{theorem:pseudo-spectral-gap-estimation-absolute}, \ref{theorem:pseudo-spectral-gap-estimation-relative}, \ref{theorem:pseudo-spectral-gap-estimation-relative-arbitrary}).
In Section~\ref{section:empirical-estimation}, we define the reversible dilation of a chain, explore its properties, and
set out to
improve the empirical estimation procedure and confidence intervals of \citet{pmlr-v99-wolfer19a}.
Section~\ref{section:algorithm} discusses the implementation details and computational complexity of the empirical estimation procedure of Section~\ref{section:empirical-estimation}.
Some proofs are deferred to Section~\ref{section:proofs} for readability.

\subsection{Related work}
\label{section:related-work}
 \subsubsection{Reversible setting}
 
 \citet{NIPS2015_7ce3284b} initiated the research program of estimating the absolute spectral gap of a reversible Markov chain from a single trajectory of observations. Using Hilbert space techniques, they obtained upper and lower bounds (Table~\ref{table:comparison-state-of-the-art}) on the estimation problem to multiplicative error, and designed fully empirical confidence intervals that decay at roughly $1/\sqrt{m}$. \citet{levin2016estimating} and \citet{hsu2019} later strengthened the minimax sample complexity upper bound in this setting, nearly matching the lower bound.
 \citet{pmlr-v99-wolfer19a} subsequently designed a finer lower bound that involves the precision parameter, further tightened the rate and simplified the empirical procedure of \citet{hsu2019}, by removing the need to compute a group (Drazin) inverse in order to spell out the confidence intervals.
 Finally, \citet{combes2019computationally} took a different approach and
designed a space efficient estimator by invoking power methods and upper confidence interval techniques.
 
\subsubsection{General, non-reversible setting} 
 
The first results in the non-reversible case were obtained by \citet{pmlr-v99-wolfer19a}, who gave an upper bound on the sample complexity of estimating the pseudo-spectral gap of a Markov chain and constructed fully empirical confidence intervals for it.
In Table~\ref{table:comparison-state-of-the-art}, we compare our result to the above-mentioned references within the PAC minimax framework at multiplicative error.

{\def\arraystretch{3}
\begin{table}%
    \begin{center}
    \begin{tabular}{c|c|c}
         Setting / Estimated parameter & Reversible / $\gamma_\star$ & Non-reversible / $\gamma \ps$\\\hline
        {\small \citet{NIPS2015_7ce3284b}} & $\Omega\left( \frac{\abs{\calX}}{\gamma_\star}\right)$, $\Omega\left(\frac{1}{\pimin}\right)$, $\widetilde{\bigO}\left( \frac{1}{ \pimin \gamma_\star^3 \eps^2}\right)$ & \\\hline
        \makecell{{\small \citet{levin2016estimating}} \\ {\small \citet{hsu2019}}} & $\widetilde{\bigO}\left( \frac{1}{\pi_\star \gamma_ \star\eps^2}\right)$ & \\\hline
        {\small \citet{pmlr-v99-wolfer19a}} &  $\Omega\left( \frac{\abs{\calX}}{\gamma_\star \eps^2}\right)$ & \makecell{ $\Omega\left( \frac{\abs{\calX}}{\gamma \ps \eps^2}\right)$, \\ $\widetilde{\bigO}\left(\frac{1}{\pimin \gamma^3 \ps \eps^2} + \frac{\abs{\calX}}{\pimin^2  \gamma^2 \ps \eps^2} \right)$} \\\hline
        Present work &  & \makecell{ 
         $\widetilde{\bigO} \left( \frac{1}{\pimin \gamma \ps} \right)$ (Th.~\ref{theorem:pseudo-spectral-gap-estimation-relative})
        }\\
    \end{tabular}
    \end{center}
    \caption{Comparison with existing results in the literature. $\Omega$ and $\bigO$ respectively denote lower and upper bounds on the sample complexity to constant multiplicative error. When $\eps$ appears in the bound, the multiplicative error is down to arbitrarily small $\eps$ instead of constant. The tilde notation suppresses logarithmic factors in $\abs{\calX}, \pimin^{-1}, \gamma \ps^{-1}, \eps^{-1} \delta^{-1}$.}
    \label{table:comparison-state-of-the-art}
\end{table}}

Both in the reversible and non-reversible case, we can readily obtain estimates of $\tmix$ from estimates of $\gamma_\star$ or $\gamma \ps$ and $\pimin$ 
--see \eqref{eq:absolute-spectral-gap-controls-tmix}, \eqref{eq:pseudo-spectral-gap-controls-tmix}-- and the latter quantity is also possible to estimate in $\tilde{\bigO}(1/(\pimin \gamma_\star \eps^2))$ \citep{hsu2019}, $\tilde{\bigO}(1/(\pimin \gamma \ps \eps^2))$ \citep{pmlr-v99-wolfer19a}.

\subsubsection{Direct mixing time estimation}
More recently, the line of work of \citet{pmlr-v117-wolfer20a, wolfer2022empirical}
departed from spectral methods and relied on measure contraction to tackle the harder problem of estimating the mixing time to multiplicative error, instead of the (pseudo)-relaxation time. They defined a generalized version of Dobrushin's contraction coefficient that approximates the mixing time up to universal factors, without introducing logarithmic gaps in $\pimin^{-1}$ [\eqref{eq:absolute-spectral-gap-controls-tmix}, \eqref{eq:pseudo-spectral-gap-controls-tmix}]. They obtained an upper bound of $\tilde{\bigO}(\tmix/\pimin + \Xi(P))$, where $\Xi(P)$ is a quantity that is highly instance-dependent and bounded above by $\Xi(P) \leq \tmix \abs{\calX}/\pimin$. Our results are therefore complementary to \citet{wolfer2022empirical} and compare favorably in terms of sample complexity by shaving a factor $\abs{\calX}$ in the worst case.

\subsection{Preliminaries}

\subsubsection{Notation}
Let $\calX$ be a finite space. We write $\calP(\calX)$ for the set of all probability distributions over $\calX$. 
Vectors will be written as row vectors. For $x \in \calX$, $e_x$ is the vector such that for any $x' \in \calX$, $e_x(x') = \pred{x = x'}$.
For two matrices (and in particular vectors) $A$ and $B$, $A/B$ denotes the entry-wise division, $\sqrt{A}$ denotes the entry-wise square root operation,  $A \circ B$ is their Hadamard product, and $A > 0$ means that $A$ is entry-wise positive. $\rho(A)$ and $\nrm{A}$ are respectively the spectral radius and the spectral norm of $A$. For a statistic $S \colon \calX^m \to \R$ computed from a sample $X_1, \dots, X_m$, we will write either $S$ for compactness or $S(X_1, \dots, X_m)$ when we want to emphasize that $S$ is a function of the data.

\subsubsection{Ergodic Markov chains and mixing time}
By a Markov chain $X_1, \dots X_m \sim (\mu, P)$ over state space $\calX$ with initial distribution $\mu$ and (row-stochastic) transition matrix $P$, we mean that for $(x_1, \dots, x_m) \in \calX^m$,
\begin{equation*}
    \PR{X_1 = x_1, \dots, X_m = x_m} = \mu(x_1) \prod_{t=1}^{m-1}P(x_t, x_{t+1}).
\end{equation*}
For any $k \in \mathbb{N}$ 
we further define the $k$-skipped associated Markov chain,
\begin{equation}
\label{eq:skipped-chain}
    X_{1}, X_{1 + k}, X_{1 + 2k}, \dots, X_{1 + \floor{(m-1)/k}k} \sim (\mu, P^k).
\end{equation}
Let $\pi \in \calP(\calX)$ such that $\pi P = \pi$ \footnote{When the state space is finite, such $\pi$ always exists via a Brouwer fixed-point argument.}. We say that $\pi$ is a stationary distribution for $P$.
A Markov chain is called ergodic when $P$ is primitive, i.e. there exists an integer $k$ such that $P^k > 0$.
In this case $\pi$ is unique, and the minimum stationary probability $\pi_\star = \min_{x \in \calX} \pi(x)$ is positive.
Furthermore, the chain will converge to stationarity, and we define the mixing time of $P$ as
\begin{equation*}
    \tmix \eqdef \min\set{t \geq 1 \colon \max_{\mu \in \calP(\calX)}\tv{\mu P^t - \pi} < 1/4},
\end{equation*}
where $\tv{\cdot}$ is the total variation distance.

\subsubsection{The Hilbert space \texorpdfstring{$\ell_2(\pi)$}{Ã¢â€žâ€œ2(Ãâ‚¬)} and spectral methods}
Let $P$ be the transition matrix of an ergodic Markov chain over the state space $\calX$, with stationary distribution $\pi$.
Recall that on one hand, we can view $P$ as a linear operator acting on measures from the right. In particular, for $\mu_t \in \calP(\calX)$,
$$\mu_{t} P = \mu_{t+1} \in \calP(\calX),$$
corresponds to a new distribution $\mu_{t+1}$ after one step on the chain.
What is more, we can simultaneously view $P$
as a linear operator acting on real functions over $\calX$ from the left, where for $f \in \R^\calX$,
$Pf$ can be interpreted as the expected values of $f$ after a single step on the chain,
$$Pf = \left(\E[(X_t,X_{t+1}) \sim (e_x, P)]{f(X_{t+1})} \right)_{x \in \calX}\trn.$$
Since $P$ is irreducible, $\pi > 0$, and
for two functions $f,g \in \R^\calX$, we can construct the inner-product
\begin{equation*}
    \innerpi{f}{g} \eqdef \sum_{x \in \calX} f(x)g(x) \pi(x).
\end{equation*}
We write $\nrm{f}_{\pi} \eqdef \sqrt{ \sum_{x \in \calX} f(x)^2 \pi(x)}$ for its associated norm, and denote $\ell_2(\pi) = (\R^\calX, \innerpi{\cdot}{\cdot})$ the resulting Hilbert space. The induced norm of an operator $A \colon \ell_2(\pi) \to \ell_2(\pi)$ is then defined as
\begin{equation*}
    \nrm{A}_{\pi} \eqdef \sup_{f \in \ell_2(\pi) \colon \nrm{f}_{\pi} = 1} \nrm{A f}_{\pi}.
\end{equation*}
By the Riesz representation theorem, there exists a unique adjoint operator $P^\star$ such that for any $f,g \in \ell_2(\pi)$,
\begin{equation*}
    \innerpi{Pf}{g} =  \innerpi{f}{P^\star g},
\end{equation*}
which in the context of Markov chains is called the time reversal of $P$, and can be expressed as $P^\star(x,x') = \pi(x')P(x', x)/\pi(x)$.
\paragraph{Reversible setting}
When $P$ is self-adjoint, the chain is said to be reversible and satisfies the so-called detailed balance equation,
\begin{equation*}
    \pi(x)P(x,x') = \pi(x')P(x',x).
\end{equation*}
In this case, the spectrum of $P$, denoted $\sigma(P) = \set{\lambda_x}_{x \in \calX}$ is real, and the inverse of the absolute spectral gap \begin{equation}
\label{definition:absolute-spectral-gap}
\gamma_\star \eqdef 1 - \max \set{ \lambda_x \colon x \in \calX, \abs{\lambda_x} \neq 1 },
\end{equation}
termed relaxation time,
is known to effectively control (see e.g. \citet{levin2009markov}) the convergence of $P$ to stationarity,
\begin{equation}
\label{eq:absolute-spectral-gap-controls-tmix}
    \left(\frac{1}{\gamma_\star} - 1\right)\log 2 \leq \tmix \leq \frac{1}{\gamma_\star} \log \frac{4}{\pimin}.
\end{equation}
Self-adjointness on $\ell_2(\pi)$ also leads to a notion of a positive semi-definite matrix.
We denote by $\succeq_\pi$ the Loewner order with respect to $\ell_2(\pi)$. By that we mean that $A \succeq_\pi 0$ if and only if for any $f \in \ell_2(\pi)$, it holds that
$$\innerpi{A f}{f} \geq 0,$$
and we write $B \succeq_\pi A$ whenever $B - A \succeq_\pi 0$.
It will also be convenient to introduce the matrices
\begin{equation}
    \label{eq:laplacian-and-memoryless-stationary}
    \Dpi \eqdef \diag(\pi), \qquad
    L \eqdef \Dpi^{1/2} P 
\Dpi^{-1/2}, \qquad \Pi \eqdef 1 \trn \pi.
\end{equation}
Finally, we state the following well-known fact whose proof is provided in Section~\ref{section:proofs} for convenience of the reader.

\begin{lemma}
\label{lemma:technical-correspondence}
Let $P$ be ergodic with stationary distribution $\pi$. When $P$ is reversible,
\begin{equation*}
     1 - \gamma_\star(P) = \rho(P - \Pi) = \nrm{P - \Pi}_{\pi}.
\end{equation*}
\end{lemma}

\paragraph{General, non-reversible setting}
When $P$ is not reversible, $\sigma(P)$ may not be real and only the lower bound in \eqref{eq:absolute-spectral-gap-controls-tmix} remains true\footnote{In this case $\abs{\cdot}$ is interpreted as the complex norm in the definition \eqref{definition:absolute-spectral-gap} of $\gamma_\star$. }. 
\citet{fill1991eigenvalue} introduced the multiplicative reversiblization $P^\star P$ of $P$ and showed how to recover an upper bound on $\tmix$ in terms of $\gamma_\dagger(P) \eqdef \gamma_\star(P^\star P)$, albeit largely pessimistic.
\citeauthor{paulin2015concentration} later generalized this idea by introducing the pseudo-spectral gap,
\begin{equation}
\label{eq:pseudo-spectral-gap}
    \gamma \ps \eqdef \max_{k \in \N} \set{ \frac{1}{k}\gamma_\dagger\left(P^k\right)}.
\end{equation}
and showed that it controls \citep[Proposition~3.4]{paulin2015concentration} the mixing time up to a logarithmic factor 
\begin{equation}
\label{eq:pseudo-spectral-gap-controls-tmix}
    \frac{1}{2 \gamma \ps} \leq \tmix \leq \frac{1}{\gamma \ps} \log \frac{4e}{\pimin},
\end{equation}
making it a suitable spectral proxy in the non-reversible setting. 
Following Lemma~\ref{lemma:technical-correspondence}, the absolute spectral gap of the multiplicative reversiblization of a skipped chain can be conveniently re-expressed as follows.
\begin{corollary}[to Lemma~\ref{lemma:technical-correspondence}]
\label{corollary:technical-correspondence-skipped}
Let $P$ be ergodic with stationary distribution $\pi$.
For any $k \in \N$,
\begin{equation*}
    1 - \gamma_\dagger(P^k) = \nrm{(P^\star - \Pi)^k(P - \Pi)^k}_{\pi}.
\end{equation*}
\end{corollary}
\begin{proof}
From Lemma~\ref{lemma:technical-correspondence},
$$\gamma_\dagger(P^k) = 1 - \nrm{(P^\star)^{k}P^{ k} - \Pi}_{\pi}.$$
Now, note that
$$\Pi P = P \Pi = \Pi = \Pi P^\star = P^\star \Pi,$$
and that $\Pi^k = \Pi$,
from which it follows that
$$(P^\star - \Pi)^k(P - \Pi)^k = (P^\star)^{k}P^{k} - \Pi.$$
\end{proof}
For easy reference, we will write $k \ps$ for the smallest integer that verifies $\gamma \ps = \gamma_\dagger(P^{k \ps})/k \ps$, and we will call relaxation time the inverse of the pseudo-spectral gap of $P$.

\section{Minimax estimation of \texorpdfstring{$\gamma \ps$}{the pseudo-spectral gap}}
\label{section:minimax-estimation}
We begin this section by proving several properties of the pseudo-spectral gap that we will rely upon throughout the manuscript, and that may also be of independent interest.

\subsection{Properties of \texorpdfstring{$\gamma \ps$}{the pseudo-spectral gap}}
\label{section:properties-pseudo-spectral-gap}

Let $P$ be some ergodic, possibly non-reversible transition matrix with pseudo-spectral gap $\gamma \ps$.
The first lemma reveals a sub-multiplicativity property of the $\nrm{\cdot}_\pi$ norm.

\begin{lemma}
\label{lemma:sub-multiplicativity-property}
Let $r, s \in \N$, it holds that
\begin{equation*}
    \nrm{(P^\star - \Pi)^{r+s}(P - \Pi)^{r+s}}_{\pi} \leq \nrm{(P^\star - \Pi)^{r}(P - \Pi)^{r}}_{\pi} \nrm{(P^\star - \Pi)^{s}(P - \Pi)^{s}}_{\pi}. 
\end{equation*}
\end{lemma}
\begin{proof}
We begin by verifying that
$\nrm{\cdot}_\pi$ is monotone for $\succeq_\pi$, similar to the Loewner order with respect to the usual inner product.
Let $A, B$ be linear maps over $\ell_2(\pi)$ such that $A \preceq_\pi B$, and let $f_\star \in \ell_2(\pi)$ with  $\nrm{f_\star}_\pi = 1$ satisfying $\nrm{Af_\star}_\pi = \nrm{A}_{\pi}$. 
Then
\begin{equation*}
    0 \leq \innerpi{(B - A)f_\star}{f_\star} = \innerpi{B f_\star}{f_\star} - \innerpi{A f_\star}{f_\star} = \innerpi{B f_\star}{f_\star} - \nrm{A}_\pi \leq \nrm{B}_\pi - \nrm{A}_\pi.
\end{equation*}
It is also immediate that for a linear map $A$
\begin{equation}
\label{eq:dominating-diagonal}
    \nrm{A}_\pi I \succeq_\pi A.
\end{equation}
Indeed, for $f \in \ell_2(\pi)$ with $\nrm{f}_\pi = 1$, we have by linearity
\begin{equation*}
    \innerpi{(\nrm{A}_\pi I - A) f}{f} = \nrm{A}_\pi \innerpi{f}{f} - \innerpi{ A f}{f} \geq \nrm{A}_\pi - \max_{\nrm{f}_\pi = 1}\innerpi{ A f}{f} = 0.
\end{equation*}
Let us now proceed with the following expansion
\begin{equation*}
    (P^\star - \Pi)^{r+s}(P - \Pi)^{r+s} = (P^\star - \Pi)^{r}(P^\star - \Pi)^{s}(P - \Pi)^{s}(P - \Pi)^{r}.
\end{equation*}
Let $f \in \ell_2(\pi)$. Since $(P^\star - \Pi)^{r}$ is the adjoint of $(P - \Pi)^{r}$ in $\ell_2(\pi)$, and writing $g = (P - \Pi)^{r}f$,
\begin{equation*}
\begin{split}
    &\innerpi{(P^\star - \Pi)^{r}\left[\nrm{(P^\star - \Pi)^{s}(P - \Pi)^{s}}_\pi I - (P^\star - \Pi)^{s}(P - \Pi)^{s}\right](P - \Pi)^{r}f}{f} \\
    & =  
    \innerpi{\left[\nrm{(P^\star - \Pi)^{s}(P - \Pi)^{s}}_\pi I - (P^\star - \Pi)^{s}(P - \Pi)^{s}\right]g}{g} \geq  0,
\end{split}
\end{equation*}
where the inequality follows from \eqref{eq:dominating-diagonal}.
As a result,
\begin{equation*}
    (P^\star - \Pi)^{r+s}(P - \Pi)^{r+s} \preceq_\pi \nrm{(P^\star - \Pi)^{s}(P - \Pi)^{s}}_\pi (P^\star - \Pi)^{r}(P - \Pi)^{r}.
\end{equation*}
The lemma then holds as $\nrm{\cdot}_\pi$ is monotone for $\succeq_\pi$ and homogeneous.
\end{proof}

It will also be convenient for us to relate the pseudo-spectral gap of $P$ to that of $P^p$, governing the dynamics of the $p$-skipped Markov chain defined in \eqref{eq:skipped-chain}.

\begin{lemma}
\label{lemma:pseudo-spectral-gap-skipped-correspondence}
For any $p \in \N$,
\begin{equation*}
    p \gamma \ps \left(1 - \frac{p k \ps \gamma \ps}{2}\right) < \gamma \ps^{(p)} \leq p \gamma \ps,
\end{equation*}
where $\gamma^{(p)} \ps \eqdef \gamma \ps(P^p)$ is the pseudo-spectral gap of the $p$-skipped Markov chain.
As a consequence,
\begin{equation*}
    p \leq \frac{1}{k \ps \gamma \ps} \implies \gamma \ps^{(p)} \geq p \gamma \ps /2.
\end{equation*}
\end{lemma}
\begin{proof}
For the other upper bound,
\begin{equation*}
\begin{split}
\gamma \ps(P^p) &= \max_{k \in \N} \set{ \frac{1}{k} \gamma \left((P^{pk})^{\star} P^{pk}\right)} = p\max_{k \in \N} \set{ \frac{1}{pk} \gamma \left((P^{pk})^{\star} P^{pk}\right)}\\
&\leq p \max_{r \in \N} \set{ \frac{1}{r} \gamma \left((P^{r})^\star P^{r} \right)} = p \gamma \ps(P), \\
\end{split}
\end{equation*}
where the inequality follows from taking the maximum over a larger set.
For the lower bound,
\begin{equation*}
\begin{split}
    \gamma \ps(P^p) &= \max_{k \in \N} \set{ \frac{\gamma_\dagger\left(P^{pk}\right)}{k} } \geq  \frac{\gamma_\dagger\left(P^{pk \ps}\right)}{k \ps} \stackrel{(i)}{=}  \frac{1 - \nrm{(P-\Pi)^{pk \ps}(P^\star-\Pi)^{p k \ps}}_\pi}{k \ps} \\
    &\stackrel{(ii)}{\geq}  \frac{1 - \nrm{(P-\Pi)^{k \ps}(P^\star-\Pi)^{k \ps}}_\pi^p}{k \ps}  =  \frac{1 - \left( 1 - \gamma_\dagger\left(P^{k \ps}\right) \right)^p}{k \ps}  \\
    &= \frac{1 - (1 - k \ps \gamma \ps)^p}{k \ps} \stackrel{(iii)}{\geq} p \gamma \ps   \left(1 - \frac{p k \ps \gamma \ps}{2} \right),\\
\end{split}
\end{equation*}
where $(i)$ is Corollary~\ref{corollary:technical-correspondence-skipped}, $(ii)$ stems from the sub-multiplicativity property of Lemma~\ref{lemma:sub-multiplicativity-property}, and $(iii)$ is a consequence of the auxiliary
Lemma~\ref{lemma:taylor-type-inequality}.
\end{proof}
Finally, when the skipping rate is larger than the order of $\gamma^{-1}  \ps$, the pseudo-spectral gap of the skipped chain is larger than a universal constant. 

\begin{lemma}
\label{lemma:termination-inequality}
It holds that for any $p \in \N$,
\begin{equation*}
    p \geq 2^{\ceil*{\log_2 1 /\gamma \ps}} \implies \gamma^{(p)} \ps > 1/2.
\end{equation*}
\end{lemma}
\begin{proof}
By definition,
\begin{equation*}
    \gamma^{\left(2^{\ceil{\log_2 1/{\gamma \ps}}}\right)} \ps 
    \geq \gamma_\dagger\left( P^{2^{\ceil{\log_2 1/{\gamma \ps}}}}\right).
\end{equation*}
We have from Corollary~\ref{corollary:technical-correspondence-skipped} that
\begin{equation*}
\begin{split}
\gamma_\dagger\left( P^{2^{\ceil{\log_2 1/{\gamma \ps}}}}\right) &= 1 - \nrm{(P^\star - \Pi)^{2^{\ceil{\log_2 1/{\gamma \ps}}}}(P - \Pi)^{2^{\ceil{\log_2 1/{\gamma \ps}}}}}_{\pi}. \\
\end{split}
\end{equation*}
Writing $2^{\ceil{\log_2 1/{\gamma \ps}}} = k \ps q \ps + r \ps$ with $0 \leq r \ps < k \ps$ and $q \ps = \floor*{2^{\ceil{\log_2 1/{\gamma \ps}}} / k \ps }$, it follows from the sub-multiplicativity property of Lemma~\ref{lemma:sub-multiplicativity-property} that
\begin{equation*}
\begin{split}
&\nrm{(P^\star - \Pi)^{2^{\ceil{\log_2 1/{\gamma \ps}}}}(P - \Pi)^{2^{\ceil{\log_2 1/{\gamma \ps}}} }}_{\pi} \\
\leq& \nrm{(P^\star - \Pi)^{k \ps  }(P - \Pi)^{k \ps }}_{\pi}^{q \ps} \nrm{(P^\star - \Pi)^{r \ps  }(P - \Pi)^{r \ps }}_{\pi} \\
=& \left(1 - \gamma_\dagger\left(P^{k \ps}\right)\right)^{\floor*{2^{\ceil{\log_2 1/{\gamma \ps}}} / k \ps }} (1 - \gamma_\dagger(P^{r \ps})) \\
\leq& \left(1 - \gamma_\dagger\left(P^{k \ps}\right)\right)^{\floor*{1/{(\gamma \ps k \ps)}  }}. \\
\end{split}
\end{equation*}
Moreover, from the definition of $k \ps$, it holds that $\gamma_\dagger(P^{k \ps}) = \gamma \ps k \ps \in (0, 1]$.
It remains to show that for any $t \in (0,1]$, $(1 - t)^{\floor{1/t}} < 1/2$, as illustrated in Figure~\ref{figure:technical-lemma}.
\begin{figure}[H]
\centering

\begin{tikzpicture}

\begin{axis}[
    legend cell align={left},
    width=12cm,
    height=5cm,
    axis lines = left,
    xlabel = $t$
]
\draw ({axis cs:8.48388725519182,0}|-{rel axis cs:0,0}) -- ({axis cs:8.48388725519182,0}|-{rel axis cs:0,1});

\addplot [
    domain=0:1, 
    jump mark mid,
    samples=1000, 
    style=thick,
    color=black
    ]
    {(1 - x)^(floor(1/x))};

\addlegendentry{$(1 - t)^{\floor{1/t}}$}
\addplot [
    domain=0:1, 
    samples=1000, 
    color=red,
    ]
    {(1 - x)^(1/x - 1)};
\addlegendentry{$(1 - t)^{1/t - 1}$}

\addplot [
    domain=0:1, 
    samples=1000, 
    color=blue,
    ]
    {(1 - x)};
\addlegendentry{$1 - t$}

\addplot[only marks,black] coordinates {
    (1/10,0.348678)
    (1/9,134217728/387420489)
    (1/8,5764801/16777216)
    (1/7,279936/823543)
    (1/6,15625/46656)
    (1/5,0.32768)
    (1/4,81/256)
    (1/3,8/27)
    (1/2,1/4)
    (1,0)
};

\end{axis}
\end{tikzpicture}

\caption{For $t \in (0,1]$, $(1 - t)^{\floor{1/t}} < 1/2$.}
\label{figure:technical-lemma}

\end{figure}

Suppose first that $t \in (1/2, 1]$. Then $\floor{1/t} = 1$, and $(1 - t)^{\floor{1/t}} = 1 - t \in [0, 1/2)$. Suppose now that $t \in [0, 1/2)$. We have
$(1 - t)^{\floor{1/t}} < (1 - t)^{1/t - 1}$ and $t \rightarrow (1 - t)^{1/t - 1}$ is strictly increasing on $[0, 1/2)$, with $\lim_{t \to 1/2} (1 - t)^{1/t - 1} = 1/2$, thus $(1 - t)^{\floor{1/t}} < 1/2$ on $[0, 1/2)$. Finally, at $t = 1/2$, it holds that $(1 - t)^{\floor{1/t}} = 1/4 < 1/2$, whence the claim.

\end{proof}

It is instructive to compare Lemma~\ref{lemma:pseudo-spectral-gap-skipped-correspondence} and Lemma~\ref{lemma:termination-inequality} with \citet[Proposition~2.1]{wolfer2022empirical}, and observe the existence of a gap for $p$,
$ 1/k \ps <  p \gamma \ps  < 1$,
for which the above lemmata do not provide a lower bound on $\gamma \ps^{(p)}$ independently of $\pimin$. We handle the remaining range as follows.

\begin{lemma}
\label{lemma:pseudo-spectral-gap-skipped-shim}
Suppose that $p < \gamma^{-1} \ps$.
Then $\gamma \ps(P^p) > \frac{p \gamma \ps}{2\log(4e/\pimin) + 2}$.
\end{lemma}
\begin{proof}
It is easy to verify that $\tmix(P^p) \leq \ceil{\tmix/p} \leq \tmix/p + 1$ (see e.g. \citet{wolfer2022empirical}), and from \eqref{eq:pseudo-spectral-gap-controls-tmix},
\begin{equation*}
\begin{split}
\gamma \ps(P^p) \geq \frac{1}{2 \tmix(P^p)} \geq \frac{p}{2 (\tmix + p)} \geq \frac{p \gamma \ps}{2 ( \log(4e/\pimin) + p \gamma \ps)}.
\end{split}
\end{equation*}

\end{proof}

\subsection{Estimation of \texorpdfstring{$\gamma \ps$}{the pseudo-spectral gap} to additive and multiplicative error}

For two states $x$ and $x'$,
we define the natural counting random variables
\begin{equation}
\label{equation:definition-natural-counting-random-variables}
\begin{split}
N_{x}^{(k)} &\eqdef \sum_{t=1}^{\floor{(m-1)/k}} \pred{X_{1 + k(t-1)} = x},\\
N_{x x'}^{(k)} &\eqdef \sum_{t=1}^{\floor{(m-1)/k}} \pred{X_{1 + k(t-1)} = x, X_{1 + kt} = x'}, \\
 \widehat{Q}^{(k)}  &\eqdef \frac{1}{m - 1} \sum_{x,x' \in \calX} N_{x x'}^{(k)} e_x \trn e_{x'}, \qquad \widehat{\pi}^{(k)} \eqdef \frac{1}{m - 1} \sum_{x \in \calX}  N_x^{(k)} e_x.
\end{split}
\end{equation}
We will also use the shorthand notations
\begin{equation}
\begin{split}
\label{equation:def-random-variable-convenient-notation}
\Nmax^{(k)} \eqdef \max_{x \in \calX} N_x^{(k)}, \ \Nmin^{(k)} \eqdef \min_{x \in \calX} N_x^{(k)}. \\
\end{split}
\end{equation}
Additionally, we introduce
\begin{equation}
    \label{eq:laplacian-estimator}
    \widehat{L}^{(k)} = \sum_{x,x' \in \calX} \frac{N^{(k)}_{x x'}}{\sqrt{N^{(k)}_x N^{(k)}_{x'}}} e_x \trn e_{x'}.
\end{equation}
For $k=1$, we will omit superscripts and simply write $N_{x}, N_{x x'}, \Nmax, \Nmin, \widehat{\pi}, \widehat{Q}$ and $\widehat{L}$.
The astute reader will notice that our estimator in \eqref{eq:laplacian-estimator} is not well-defined when the skipped trajectory failed to visit some states.
However, we will later show that 
this event only occurs with low probability for our sampling regime, avoiding smoothing in this section.
We first recall the estimator and upper bound for the problem of estimating $\pimin$ in multiplicative error from a single trajectory of observations.

\begin{lemma}[{\citet[Theorem~1]{pmlr-v99-wolfer19a}}]
\label{lemma:estimation-pimin-relative}
Let $\eps, \delta \in (0,1)$.
Define the statistic $\widehat{\pi}_\star \colon \calX^m \to [0,1]$ by $$\widehat{\pi}_\star \eqdef \frac{1}{m-1} \Nmin,$$
where $\Nmin$ is defined in \eqref{equation:def-random-variable-convenient-notation}.
There exists a universal constant $c$ such that the following statement is true.
Let $X_1, \dots, X_m$ be a Markov chain over the state space $\calX$ with ergodic transition matrix $P$, minimum stationary probability $\pimin$, pseudo-spectral gap $\gamma \ps$ and arbitrary initial distribution. When
$$m \geq c \frac{1}{\gamma \ps \pimin \eps^2} \log \frac{1}{\delta \pimin},$$ it holds with probability at least $1 - \delta$ that 
\begin{equation*}
    \abs{\widehat{\pi}_\star(X_1, \dots, X_m) - \pimin} \leq \eps \pimin.
\end{equation*}
\end{lemma}

Our first new result in this subsection quantifies the necessary trajectory length to estimate 
the matrix $L$ associated to $P$ --see \eqref{eq:laplacian-and-memoryless-stationary}-- with respect to the spectral norm. While \citet{hsu2019} leveraged reversibility on multiple occasions
for proving a similar claim, we show that this assumption is unnecessary,
hence the result holds more generally for ergodic chains with $\gamma \ps$ in lieu of $\gamma_\star$, modulo an extra cost in the logarithmic factors.

\begin{lemma}
\label{lemma:learn-laplacian-spectral-norm}

Let $\eps, \delta \in (0,1)$.
For the estimator $\widehat{L}$ defined in \eqref{eq:laplacian-estimator}, 
there exists a universal constant $c$ such that the following statement is true.
Let $X_1, \dots, X_m$ be a Markov chain over the state space $\calX$ with ergodic transition matrix $P$, minimum stationary probability $\pimin$, pseudo-spectral gap $\gamma \ps$, arbitrary initial distribution and $L$ matrix associated to $P$ defined in \eqref{eq:laplacian-and-memoryless-stationary}. When
\begin{equation*}
    m \geq \frac{c}{\gamma \ps \pimin \eps^2} \log \frac{1}{\pimin \delta} \log \frac{1}{\gamma \ps \pimin \delta \eps} ,
\end{equation*}
it holds --regardless of reversibility-- with probability at least $1 - \delta$ that 
\begin{equation*}
    \nrm{\widehat{L}(X_1, \dots, X_m) - L} \leq \eps.
\end{equation*}
\end{lemma}
\begin{proof}
The proof of this lemma is deferred to Section~\ref{section:proofs-laplacian}.
\end{proof}

We proceed to estimate $\gamma \ps$, first to additive error $\eps \in (0,1)$. Following \citet{pmlr-v99-wolfer19a}, we analyze a truncated version of the empirical pseudo-spectral gap.
Namely, for some fixed $K \in \N$, we define the estimator
\begin{equation}
    \label{eq:pseudo-spectral-gap-estimator-absolute}
    \widehat{\gamma} \psK \eqdef \max_{k \in [K]} \set{ \frac{1}{k}\gamma_\star \left(\left(\widehat{L}^{(k) }\right)\trn\widehat{L}^{(k) }\right)}.
\end{equation}

We immediately note that since the above estimator is computed for only a prefix $[K]$ of the integers, it will naturally incur an approximation error.
Our next result takes the approach of \citet{pmlr-v99-wolfer19a} to translate the improved result of Lemma~\ref{lemma:learn-laplacian-spectral-norm} to an estimation upper bound for $\gamma \ps$.

\begin{theorem}[Arbitrary additive error]
\label{theorem:pseudo-spectral-gap-estimation-absolute}
Let $\eps, \delta \in (0,1)$.
For the estimator $\widehat{\gamma} \psK$ defined in \eqref{eq:pseudo-spectral-gap-estimator-absolute} with $K = \ceil{2 /\eps}$, 
there exists a universal constant $c$ such that the following statement is true.
Let $X_1, \dots, X_m$ be a Markov chain over the state space $\calX$ with ergodic transition matrix $P$, minimum stationary probability $\pimin$, pseudo-spectral gap $\gamma \ps$, arbitrary initial distribution. When
\begin{equation*}
    m \geq \frac{c}{\gamma \ps \pimin \eps^2} \log \frac{1}{\pimin} \log \frac{1}{\pimin \eps \delta} \log \frac{1}{\pimin \gamma \ps \eps \delta} ,
\end{equation*}
it holds with probability at least $1 - \delta$ that 
\begin{equation*}
    \abs{\estpssg - \gamma \ps} \leq \eps.
\end{equation*}
\end{theorem}

\begin{proof}

\emph{Reduction to a maximum over a finite number of estimators.}

We apply the reduction of \citet{pmlr-v99-wolfer19a}.
For $K \in \mathbb{N}$, we write
$$\gamma \psK = \max_{k \in [K]} \set{ \frac{1}{k}\gamma_\dagger(P^k)},$$
and obtain
\begin{equation*}
\begin{split}
\abs{\widehat{\gamma} \psK - \gamma \ps} &\stackrel{(i)}{\leq} \frac{1}{K} + \abs{\widehat{\gamma} \psK - \gamma \psK} \\
&\stackrel{(ii)}{\leq} \frac{1}{K} + \max_{k \in [K]} \set{  \frac{1}{k} \abs{ \gamma_\star \left( \left(\widehat{L}^{(k)}\right) \trn \widehat{L}^{(k)} \right) - \gamma_\star \left( \left(L^{k}\right) \trn L^{k} \right) } } \\
&= \frac{1}{K} + \max_{k \in [K]} \set{  \frac{1}{k} \abs{ \max \set{\widehat{\lambda}^{(k) \downarrow}_2, \widehat{\lambda}^{(k) \downarrow}_{\abs{\calX}}} -  \max \set{\lambda^{(k), \downarrow}_2, \lambda^{(k) \downarrow}_{\abs{\calX}}} } } \\
&\leq \frac{1}{K} + \max_{k \in [K]} \set{  \frac{1}{k} \abs{ \max \set{ \abs{ \widehat{\lambda}^{(k) \downarrow}_2 - \lambda^{(k), \downarrow}_2}, \abs{ \widehat{\lambda}^{(k) \downarrow}_{\abs{\calX}} - \lambda^{(k) \downarrow}_{\abs{\calX}}  } }  } } \\
&\stackrel{(iii)}{\leq} \frac{1}{K} + \max_{k \in [K]} \set{  \frac{1}{k}  \nrm{ \left(\widehat{L}^{(k)}\right) \trn \widehat{L}^{(k)} -  \left(L^{k}\right) \trn L^{k}}},
\end{split}
\end{equation*}
where for simplicity we introduced  $\left(\widehat{\lambda}^{(k) \downarrow}_x\right)_{x \in \calX}$ and $\left(\lambda^{(k) \downarrow}_x\right)_{x \in \calX}$ the ordered spectra of $ \left(\widehat{L}^{(k)}\right) \trn \widehat{L}^{(k)}$ and $\left(L^{k}\right) \trn L^{k}$ respectively.
Inequality $(i)$ is a consequence of Markov kernels having a unit spectral radius, 
$(ii)$ is entailed by sub-additivity of the $\nrm{\cdot}_\infty$ norm in $\R^K$ and by matrix similarity, 
$(iii)$ is Weyl's inequality \citep[Corollary~4.9]{stewart1990matrix} for symmetric matrices.
We follow up by observing that by sub-additivity and sub-multiplicativity of the spectral norm, together with $\|L^k\| = 1$,
\begin{equation*}
\begin{split}
    \nrm{ \left(\widehat{L}^{(k)}\right) \trn \widehat{L}^{(k)} -  \left(L^{k}\right) \trn L^{k}} &= \nrm{ \left(\widehat{L}^{(k)}\right) \trn \widehat{L}^{(k)} - \left(\widehat{L}^{(k)}\right) \trn L^{k} + \left(\widehat{L}^{(k)}\right) \trn L^{k} - \left(L^{k}\right) \trn L^{k}} \\
    &\leq \nrm{ \left(\widehat{L}^{(k)}\right) \trn \widehat{L}^{(k)} - \left(\widehat{L}^{(k)}\right) \trn L^{k}} + \nrm{ \left(\widehat{L}^{(k)}\right) \trn L^{k} - \left(L^{k}\right) \trn L^{k}} \\
    &\leq \nrm{\left(\widehat{L}^{(k)}\right) \trn} \nrm{  \widehat{L}^{(k)} -  L^{k}} + \nrm{ \left(\widehat{L}^{(k)}\right) \trn  - \left(L^{k}\right) \trn} \nrm{L^{k}} \\
    &= 2 \nrm{  \widehat{L}^{(k)} -  L^{k}}.\\
\end{split}
\end{equation*}

\emph{Putting everything together.}

We set $K = \ceil{\eps/2}$ to bound the approximation error. 
Invoking Lemma~\ref{lemma:learn-laplacian-spectral-norm} for each $k$-skipped chain,
for 
\begin{equation*}
    m/k \geq \frac{c}{\gamma^{(k)} \ps \pimin k^2 \eps^2} \log \frac{K}{\pimin \delta} \log \frac{K}{\gamma \ps^{(k)} \pimin \delta k \eps} ,
\end{equation*}
with probability $1 - \delta/K$,
$\nrm{\widehat{L}^{(k)} - L^k} \leq \eps k / 4$. An application of Lemma~\ref{lemma:pseudo-spectral-gap-skipped-shim} followed by a union bound finishes proving the theorem.

\end{proof}

Our next objective is to obtain upper bounds for the estimation problem to multiplicative error $\eps$, for which we design two different procedures.
The first is inspired from a technique of \citet{levin2016estimating}, which consists in amplifying the estimator in Theorem~\ref{theorem:pseudo-spectral-gap-estimation-absolute} and which matches the minimax rate lower bound when $\eps > 5$. 
Namely, we define 
\begin{equation}
\label{eq:pseudo-spectral-gap-estimator-relative}
    \widehat{\gamma}\ps \eqdef \widehat{\gamma}^{(\widehat{K}_\star)}\ps / \widehat{K}_\star,
\end{equation}
where
$$\log_2 \widehat{K}_\star \eqdef \argmin_{p \in \set{0, 1, 2, \dots}} \set{ \widehat{\gamma}^{(2^p)} \ps > 3/8 },$$
and for $k \in \N$, $$\widehat{\gamma}^{(k)}\ps = \widehat{\gamma} \psK \left( X_{1}, X_{1 + k}, X_{1 + 2k}, \dots, X_{1 + \floor{(m-1)/k}k} \right)$$ 
is the estimator for $\gamma \ps^{(k)} \eqdef \gamma \ps (P^{k})$ to additive error defined in \eqref{eq:pseudo-spectral-gap-estimator-absolute}, with $K = 16$, and estimated from the $k$-skipped Markov chain.

\begin{theorem}[Constant multiplicative error]
\label{theorem:pseudo-spectral-gap-estimation-relative}

Let $\delta \in (0,1)$.
 For the amplified estimator $\widehat{\gamma} \ps$ defined in \eqref{eq:pseudo-spectral-gap-estimator-relative}, 
there exists a universal constant $c$ such that the following statement is true.
Let $X_1, \dots, X_m$ be a Markov chain over the state space $\calX$ with ergodic transition matrix $P$, minimum stationary probability $\pimin$, pseudo-spectral gap $\gamma \ps$, arbitrary initial distribution. When
\begin{equation*}
    m \geq c 
 \frac{1}{\gamma \ps \pimin} \log^2 \frac{1}{\pimin} \log \frac{\log 1/ \gamma \ps}{\pimin \delta} \log \frac{1}{\pimin \gamma \ps \delta},
\end{equation*}
it holds with probability at least $1 - \delta$ that 
\begin{equation*}
    \abs{\widehat{\gamma}\ps -\gamma \ps} \leq 5 \gamma \ps.
\end{equation*}
\end{theorem}

\begin{proof}

The proof adopts the strategy in \citet{wolfer2022empirical}, that adapts the amplification method of \citet{levin2016estimating} and \cite{hsu2019} to the case where it is not possible to have an exact correspondence between the relaxation time of a chain and its skipped version.
In this method, we first bound the probability for a different notion of distance (see \eqref{eq:relative-confidence-stronger-definition}) and will rely on tools we developed in Section~\ref{section:properties-pseudo-spectral-gap}.
First, fix $\eta \in (0,1/8)$.
We consider the estimator defined in \eqref{eq:pseudo-spectral-gap-estimator-relative}
where
$$\log_2 \widehat{K}_\star = \argmin_{p \in \set{0, 1, 2, \dots}} \set{ \widehat{\gamma}^{(2^p)} \ps > 1/4 + \eta },$$
and for $k \in \N$, $\widehat{\gamma}^{(k)}\ps = \widehat{\gamma} \psK \left( X_{1}, X_{1 + k}, X_{1 + 2k}, \dots, X_{1 + \floor{(m-1)/k}k} \right)$ 
is the estimator for $\gamma \ps^{(k)} \eqdef \gamma \ps (P^{k})$ to additive error defined in \eqref{eq:pseudo-spectral-gap-estimator-absolute}, with $K = \ceil{2/\eta}$, estimated from the $k$-skipped Markov chain. 
Let us define the events for $k \in \N$,
$$g(k, \eta) \eqdef \set{ \abs{\widehat{\gamma}^{(k)}\ps - \gamma^{(k)}\ps} \leq \eta }, \qquad G(\eta) \eqdef \bigcap_{p = 0}^{\overline{p}} g(2^p, \eta),$$
where $\overline{p} \eqdef \ceil{\log_2 1 / \gamma \ps}$. 
We will prove that for some universally fixed $\alpha \geq 5$,
\begin{equation}
\label{eq:relative-confidence-stronger-definition}
    \PR{\max \set{\frac{\widehat{\gamma} \ps}{\gamma\ps}, \frac{\gamma\ps}{\widehat{\gamma} \ps}} > 1 + \alpha} \leq \delta.
\end{equation}
Let us set
\begin{equation*}
    m \geq c \max_{p \in \set{0, 1, \dots, \overline{p}}} \set{
 2^p \frac{1}{\gamma^{(2^p)} \ps \pimin \eta^2} \log \frac{1}{\pimin} \log \frac{\log 1/ \gamma \ps}{\pimin \eta \delta} \log \frac{\log 1/ \gamma \ps}{\pimin \gamma^{(2^p)} \ps \eta \delta} },
\end{equation*}
and let us first assume that we are on the event $G(\eta)$.
We verify that the algorithm terminates before $\widehat{K}_\star = 2^{\overline{p}}$.
From Lemma~\ref{lemma:termination-inequality}, it holds that
\begin{equation*}
    \gamma^{\left(2^{\ceil{\log_2 1 /\gamma \ps}}\right)} \ps > 1/2.
\end{equation*}
Being on $G(\eta)$,
$$\widehat{\gamma}^{\left(2^{\ceil{\log_2 1 /\gamma \ps}}\right)} \geq \gamma^{\left(2^{\ceil{\log_2 1 /\gamma \ps}}\right)} - \eta > 1/2 - \eta > 1/4 + \eta,$$
where the last inequality stems from the choice of the range for $\eta$. It follows that the algorithm stops at power $2^{\overline{p}}$ at the latest.
For $\alpha \geq 2/(1/4 + \eta) - 1$ and $\alpha > 4 \eta$, on the event $G(\eta)$, we now prove that
\begin{equation*}
    \max \set{ \frac{\widehat{\gamma} \ps  }{\gamma\ps}, \frac{\gamma\ps}{\widehat{\gamma} \ps } } \leq 1 + \alpha.
\end{equation*}
On one hand, since
$$\widehat{K}_\star \leq 2^{\ceil{\log_2 1/ \gamma \ps}} \leq 2^{1 + \log_2 1/ \gamma \ps} = 2/\gamma \ps,$$ we have from the stopping rule that
\begin{equation*}
   \frac{\gamma\ps}{\widehat{\gamma} \ps } =
   \frac{\widehat{K}_\star \gamma\ps}{\widehat{\gamma}^{(\widehat{K}_\star)} \ps} \leq \frac{2}{ \widehat{\gamma}^{(\widehat{K}_\star)} \ps} \leq \frac{2}{1/4 + \eta} \leq 1 + \alpha.
\end{equation*}
On the other hand, from Lemma~\ref{lemma:pseudo-spectral-gap-skipped-correspondence}, for any $k \in \N$, $\gamma \ps^{(k)} \leq k \gamma \ps$,
and thus
\begin{equation*}
    \frac{\widehat{\gamma} \ps }{\gamma\ps} =
   \frac{\widehat{\gamma}^{(\widehat{K}_\star)} \ps}{\widehat{K}_\star \gamma\ps} \leq \frac{\widehat{\gamma}^{(\widehat{K}_\star)} \ps}{\gamma^{(\widehat{K}_\star)} \ps} = 1 + \frac{\widehat{\gamma}^{(\widehat{K}_\star)} \ps - \gamma^{(\widehat{K}_\star)} \ps}{\gamma^{(\widehat{K}_\star)} \ps} \stackrel{(a)}{\leq} 1 + \frac{\eta}{\gamma^{(\widehat{K}_\star)} \ps} \stackrel{(b)}{\leq} 1 + 4 \eta < 1 + \alpha,
\end{equation*}
where inequalities $(a), (b)$ follow from being on $G(\eta)$ and $(b)$ also requires the stopping condition.
The event inclusion follows,
\begin{equation*}
    \set{\max \set{ \frac{\widehat{\gamma} \ps  }{\gamma\ps}, \frac{\gamma\ps}{\widehat{\gamma} \ps } } > 1 + \alpha} \subset G(\eta)^\complement,
\end{equation*}
and to prove \eqref{eq:relative-confidence-stronger-definition}, we are left with bounding the probability of $G(\eta)^\complement$ occurring.
By the union bound and Theorem~\ref{theorem:pseudo-spectral-gap-estimation-absolute}, $\PR{G(\eta)^\complement} \leq \delta$.
Moreover, for any $a,b > 0$, 
$$\abs{a/b - 1} \leq \max \set{a/b, b/a} - 1,$$ 
thus
\begin{equation*}
    \PR{\abs{ \frac{\widehat{\gamma} \ps  }{\gamma\ps} - 1 } > \alpha} \leq \PR{\max \set{ \frac{\widehat{\gamma} \ps  }{\gamma\ps}, \frac{\gamma\ps}{\widehat{\gamma} \ps } } > 1 + \alpha} \leq \PR{G(\eta)^\complement}.
\end{equation*}
Setting $\eta = 1/8$ and $\alpha = 5$ also fixes $K = 16$ and finishes proving the theorem.

\end{proof}

\begin{remark}
The ``doubling trick" applied in \citet{levin2016estimating} 
exploited the identity
\begin{equation*}
    \gamma_\star(P^k) = 1 - (1 - \gamma_\star)^k,
\end{equation*}
to achieve arbitrary precision $\eps$, but only holds only under reversibility.
\citet{wolfer2022empirical} ran into the similar technical difficulty for estimating a ``generalized contraction coefficient", which they connected to its skipped counterpart
\citep[Proposition~2.1]{wolfer2022empirical}.
The tools therein 
are insufficient
for our purposes, and we use the apparatus in Section~\ref{section:properties-pseudo-spectral-gap} to connect $\gamma \ps$ and $\gamma^{(k)}\ps$, particularly the sub-multiplicativity property of Lemma~\ref{lemma:sub-multiplicativity-property}.
\end{remark}

While we 
presume
that for most applications, Theorem~\ref{theorem:pseudo-spectral-gap-estimation-relative} 
will suffice
for the
practitioner, it is theoretically interesting to fully understand to dependency in $\eps$ of the sample complexity.
Our second estimator computes $\gamma_\dagger(P^k)$ for $k$ up to some adaptive $\hat{K}$ and can achieve arbitrarily small error $\eps$, albeit with a trajectory length that is inversely proportional to the cube of $\gamma \ps$.

\begin{theorem}[Arbitrary multiplicative error]
\label{theorem:pseudo-spectral-gap-estimation-relative-arbitrary}
Let $\eps \in (0, 5)$ and $\delta \in (0,1)$.
For the estimator $\widehat{\gamma} \pshatK$ 
defined in
\eqref{eq:pseudo-spectral-gap-estimator-absolute}
with adaptive 
$$\widehat{K}(X_1, \dots, X_m) = \ceil*{(\Nmin/\eps)^{1/3}},$$
there exists a universal constant $c$ such that the following statement is true.
Let $X_1, \dots, X_m$ be a Markov chain over the state space $\calX$ with ergodic transition matrix $P$, minimum stationary probability $\pimin$, pseudo-spectral gap $\gamma \ps$, arbitrary initial distribution. When
\begin{equation*}
    m \geq \frac{c}{\pimin \gamma^3 \ps \eps^2} \log \frac{1}{\pimin} \log^2 \frac{1}{\pimin \gamma \ps \eps \delta},
\end{equation*}
it holds with probability at least $1 - \delta$ that 
\begin{equation*}
    \abs{\widehat{\gamma}\pshatK -\gamma \ps} \leq \eps \gamma \ps.
\end{equation*}
\end{theorem}

\begin{proof}

Let us set
$$\widehat{K}(X_1, \dots, X_m) = \ceil*{\left( \Nmin / \eps \right)^{1/3}},$$
i.e. the prefix $[K]$ of integers we explore will depend on the data.
We decompose as follows,
\begin{equation*}
    \begin{split}
        \abs{\widehat{\gamma}\pshatK - \gamma \ps} \leq \abs{\widehat{\gamma}\pshatK - \widehat{\gamma}\psKeps} + \abs{\widehat{\gamma}\psKeps - \gamma \psKeps} + \abs{\gamma \psKeps - \gamma \ps},
    \end{split}
\end{equation*}
where we wrote $K_{\eps, \gamma \ps} = \ceil{3/(\eps \gamma \ps)}$.
The third summand is immediately smaller than $\eps \gamma \ps / 3$.
The second term is handled similar to Theorem~\ref{theorem:pseudo-spectral-gap-estimation-absolute} at precision $\eps \gamma \ps /3$ and confidence $1 - \delta /2$.
On the event where $\widehat{K} \geq K_{\eps, \gamma \ps}$,
\begin{equation*}
    \abs{\widehat{\gamma}\pshatK - \widehat{\gamma} \psKeps} \leq \max_{K_{\eps, \gamma \ps < k \leq K}} \set{ \frac{\gamma \left(\left(\widehat{L}^{(k) }\right)\trn\widehat{L}^{(k) }\right)}{k}} \leq \frac{1}{K_{\eps, \gamma \ps}} \leq \frac{\eps \gamma \ps}{3},
\end{equation*}
and so it suffices to bound
\begin{equation*}
    \PR{\widehat{K} \leq K_{\eps, \gamma \ps}} \leq \PR{\widehat{K} \leq \frac{3}{\eps \gamma \ps}}
\end{equation*}
From the definition of $\widehat{\pi}_\star$ (Lemma~\ref{lemma:estimation-pimin-relative}) and when $m \geq \frac{54}{\pimin \gamma \ps^3 \eps^2}$,
\begin{equation*}
\begin{split}
    \PR{\widehat{K} \leq \frac{3}{\eps \gamma \ps}} &= \PR{ \ceil*{\left(\frac{\Nmin}{\eps}\right)^{1/3}} \leq \frac{3}{\eps \gamma \ps}} \leq \PR{ m \widehat{\pi}_\star \leq \frac{27}{\eps^2 \gamma^3 \ps}} \\
    &\leq \PR{ \frac{54}{\pimin \gamma \ps^3 \eps^2} \widehat{\pi}_\star \leq \frac{27}{\eps^2 \gamma^3 \ps}} = \PR{  \widehat{\pi}_\star \leq \frac{\pimin}{2}}\\
    &\leq  \PR{\abs{\widehat{\pi}_\star - \pi_\star} > \frac{\pimin}{2}}.
\end{split}
\end{equation*}
The claim then follows from Lemma~\ref{lemma:estimation-pimin-relative} at multiplicative error $\eps = 1/2$ and an application of the union bound.

\end{proof}

\begin{oproblem}
Close the dependency gap in $\eps$ and $\pimin$ between the minimax upper and lower bounds.
\end{oproblem}

\section{Empirical estimation of \texorpdfstring{$\gamma \ps$}{the pseudo-spectral gap}}
\label{section:empirical-estimation}

The previous section analyzed the estimation problem from a minimax perspective. While the obtained upper bounds
enable us to understand its statistical complexity, they also come with a serious caveat as they depend on the --a priori unknown-- $\gamma \ps$ and $\pimin$. In this section, we set out the task of constructing non-trivial empirical confidence intervals for $\gamma \ps$. Namely, for $\delta \in (0,1)$, we wish to find an interval $\widehat{I}(X_1, \dots, X_m) \subset (0,1]$ such that with probability at least $1 - \delta$, it holds that $\gamma \ps \in \widehat{I}$, and $|\widehat{I}| \stackrel{a.s.}{\rightarrow} 0$.

\subsection{Reversible dilation}
\label{section:reversible-dilation}
From a non-reversible transition matrix $P$ there are multiple ways of constructing a reversible one.

\subsubsection{Known reversiblizations}

We 
mentioned 
above
the multiplicative reversiblization of \citet{fill1991eigenvalue}.
Another common operation is the additive reversiblization of $P$, that corresponds to the m-projection \citep[Theorem~7]{wolfer2021information} of $P$ onto the manifold of reversible transition matrices \footnote{For simplicity, we assume for the expressions of $P_m$ and $P_e$ that $P$ is full support.},
$
    (P + P^\star)/2 = P_m \eqdef \argmin_{\bar{P} \in \calW \rev} D(P \| \bar{P}),
$
where $D$ is the Kullback-Leibler divergence between the two Markov processes governed by $P$ and $\bar{P}$, and $\calW \rev$ is the family of reversible Markov transition matrices.
For instance, the estimator for $\gamma_\star$ in \citet{hsu2019} is essentially taken to be the absolute spectral gap of the additive reversiblization of the tally matrix.
The dual of the reversible m-projection is the reversible e-projection, given by
$
    P_e \eqdef \argmin_{\bar{P} \in \calW \rev} D(\bar{P} \| P) = \diag (v) (P^\star \circ P) \diag(v)^{-1} /\rho,
$
where $\rho$ and $v$ are the Perron-Frobenius root and associated right eigenvector of $P^\star \circ P$. 
The reversiblizations $P_m$ and $P_e$ are known to satisfy Pythagorean inequalities with respect to the Kullback-Leibler divergence between the Markov processes  \citep[Theorem~7]{wolfer2021information}.
Finally, we mention the Metropolis-Hastings reversiblization proposed by \citet{choi2020metropolis}.
Additive and multiplicative reversiblizations have both proven to been instrumental for studying the mixing time of general ergodic chains; the former for continuous-time chains, and the latter in a discrete-time setting \citep{montenegro2006mathematical, fill1991eigenvalue}. The pseudo-spectral gap of \citet{paulin2015concentration} is a generalization of the multiplicative reversiblization of \citeauthor{fill1991eigenvalue}, such that a naive plug-in approach for it suggests multiplying empirical transition matrices with their respective time reversal --a computationally expensive operation that was implemented in \citet{pmlr-v99-wolfer19a}.

\subsubsection{New reversiblization: the reversible dilation} We now introduce the notion of the reversible dilation $\mathscr{S}(P)$ of a chain,
closely related to its multiplicative counterpart in that it shares strong spectral similarities 
(see Lemma~\ref{lemma:multiplicative-and-dilation-similar-spectrum}) whose generalized version 
still controls the mixing time \eqref{eq:dilated-pseudo-spectral-gap-controls-mixing-time} in a 
discrete-time setting, and that will allow us to substantially reduce the computational cost of our estimator. 
Indeed, when we know $\pi$, at the price of embedding the Markov operator in a space
of twice the dimension,
we avoid the need 
for matrix multiplication. The ideas in this section take inspiration from the concept of 
self-adjoint dilations borrowed from \citet{paulsen2002completely} that have also 
recently found their use in the field of matrix concentration \citep{tropp2012user}. 
Moreover, we argue in Lemma~\ref{lemma:pssg-reversible} that a pseudo-spectral gap 
whose definition is based on reversible dilations instead of multiplicative reversiblizations is a more
natural generalization of the absolute spectral gap of a reversible Markov operator.

\begin{definition}[Reversible dilation]
\label{definition:reversible-dilation}
Let $P$ be the transition matrix of a Markov chain over the state space $\calX$, with stationary distribution $\pi$. We define the reversible dilation of $P$ to be the square matrix of size $2 \abs{\calX}$,
$$\mathscr{S}_{\pi}(P) \eqdef \begin{pmatrix} 0 & P \\ P^\star & 0 \end{pmatrix},$$
where $P^\star$ is the adjoint of $P$ in $\ell^2(\pi)$.
\end{definition}
For convenience, we also extend the $\mathscr{S}$ operator to
all real matrices via
$\mathscr{S}(A) = \begin{psmallmatrix} 0 & A \\ A \trn & 0 \end{psmallmatrix}$, to denote
the self-adjoint dilation with respect to the usual inner product in $\R^\calX$.

\begin{proposition}
\label{proposition:properties-reversible-dilation}
Let $P$ be a stochastic matrix with stationary distribution $\pi$, then the reversible dilation $\mathscr{S}_{\pi}(P)$ in Definition~\ref{definition:reversible-dilation} verifies the following properties:
\begin{enumerate}[$(i)$]
	\item $\mathscr{S}_{\pi}(P)$ is a $(2\abs{\calX}) \times (2\abs{\calX})$ stochastic matrix.
	\item The concatenated and normalized vector $\frac{1}{2} \begin{pmatrix} \pi & \pi  \end{pmatrix}$ is a stationary distribution for $\mathscr{S}_{\pi}(P)$.
	\item $\mathscr{S}_{\pi}(P)$ is 2-periodic.
	\item $\mathscr{S}_{\pi}(P)$ is reversible with respect to $\frac{1}{2} \begin{pmatrix} \pi & \pi  \end{pmatrix}$.
\end{enumerate}
\end{proposition}
\begin{proof}
Statements $(i)$ and $(ii)$ follow immediately from the properties of $P^\star$.
Notice that for any $k \in \N$,
\begin{equation*}
\begin{split}
\mathscr{S}_{\pi}(P)^{2k} &= \begin{pmatrix} (P P^\star)^k & 0 \\ 0 & (P^\star P )^k \end{pmatrix}, \\
\mathscr{S}_{\pi}(P)^{2k+1} &= \begin{pmatrix}  0 & (P P^\star)^k P  \\  (P^\star P )^k P^\star & 0 \end{pmatrix},
\end{split}
\end{equation*}
hence $(iii)$ holds.
To prove the reversibility property $(iv)$ it suffices to verify with a direct computation that 
$$\begin{psmallmatrix} \Dpi & 0 \\  0 & \Dpi \end{psmallmatrix} \mathscr{S}_{\pi}(P) = \mathscr{S}_{\pi}(P) \trn \begin{psmallmatrix} \Dpi & 0 \\  0 & \Dpi \end{psmallmatrix}.$$ 
\end{proof}

We henceforth use the shorthand $\gamma_\ddagger = \gamma_\star(\mathscr{S}(P))$.
Proposition~\ref{proposition:properties-reversible-dilation}-$(iii)$ implies that $\mathscr{S}_{\pi}(P)$ is never ergodic and that $-1$ is an eigenvalue of $\mathscr{S}_{\pi}(P)$.
In fact, it can even be that $P$ is irreducible, while $\mathscr{S}_{\pi}(P)$ is not, as illustrated below.

\begin{example}
We borrow the matrix
$P = \left(\begin{smallmatrix}0 & 1 & 0 \\ 0 & 0 & 1 \\ 1/2 & 0 & 1/2\end{smallmatrix}\right)$  over $\calX = \set{0,1,2}$ given at \citet[Example~6.2]{montenegro2006mathematical}.
 A direct computation yields
 \begin{equation*}
    \mathscr{S}_{\pi}(P) = \begin{pmatrix}
     & \vdots & & 0 & 1 & 0 \\
    \hdots & 0 & \hdots & 0 & 0 & 1 \\
     & \vdots & & 1/2 & 0 & 1/2 \\
     0 & 0 & 1 & & \vdots & \\
     1 & 0 & 0 & \hdots & 0 & \hdots \\
     0 & 1/2 & 1/2 & & \vdots & \\
     \end{pmatrix} \in \calW(\set{0,1,2,3,4,5}),
 \end{equation*}
 with a connection graph composed of two communicating classes $A = \set{0, 4}$ and $B = \set{1,2,3,5}$.
 For any $\theta \in [0, 1]$,
 \begin{equation*}
 \pi_\theta = \left( \frac{\theta}{2} , \frac{1 - \theta}{6}, \frac{1 - \theta}{3}, \frac{1 - \theta}{6}, \frac{\theta}{2}, \frac{1 - \theta}{3}\right)    
 \end{equation*}
 is a stationary distribution for $\mathscr{S}_{\pi}(P)$.
\end{example}

What is more, the spectral gaps of the multiplicative reversiblization and reversible dilation are closely related.

\begin{lemma}
\label{lemma:multiplicative-and-dilation-similar-spectrum}
\begin{equation*}
\gamma_\ddagger \stackrel{(a)}{\leq} \gamma_\dagger \stackrel{(b)}{=} \gamma_\ddagger(2 - \gamma_\ddagger) \stackrel{(c)}{\leq} 2\gamma_\ddagger.
\end{equation*}
\end{lemma}
\begin{proof}
The proof is standard. Notice first that 
\begin{equation}
\label{equation:similarity-dilation-rescaled}
\begin{pmatrix} \Dpi^{-1/2} & 0 \\ 0 & \Dpi^{-1/2} \end{pmatrix} \begin{pmatrix} 0 & P \\ P^\star & 0 \end{pmatrix} \begin{pmatrix} \Dpi^{1/2} & 0 \\ 0 & \Dpi^{1/2} \end{pmatrix} = \begin{pmatrix} 0 & L \\ L \trn & 0 \end{pmatrix},
\end{equation}
i.e. $\mathscr{S}(L)$
and $\mathscr{S}_{\pi}(P)$ are similar matrices, thus have identical eigen-systems.
Also from a matrix similarity argument, $\gamma_\dagger(P) = \gamma(L \trn L)$. Denote by $\lambda_1 , \lambda_2,\dots, \lambda_{\abs{\calX}}$ the eigenvalues of $L \trn L$, and by $\mu_1, \mu_2, \dots, \mu_{2 \abs{\calX}}$ the eigenvalues of $\mathscr{S}(L)$.
Since $L \trn L$ is positive semi-definite hermitian,
the $\set{\lambda_x}_{x \in \calX}$ are all real and non-negative  (see e.g. \citet{fill1991eigenvalue}).
We can therefore order them as $1 = \lambda^\downarrow_1 \geq \lambda^\downarrow_2 \geq \dots \geq \lambda^\downarrow_{\abs{\calX}} \geq 0$. Since $\mathscr{S}_{\pi}(P)$ is a stochastic matrix, its eigenvalues $\mu_1, \dots, \mu_{2\abs{\calX}}$ can be chosen such that $\mu_1 = 1$ and $\abs{\mu_x} \leq 1, \forall x, 1 \leq x \leq 2 
\abs{\calX}$. It is a direct consequence of
\begin{equation*}
\mathscr{S}(L)^2 = \begin{pmatrix} L L \trn & 0 \\ 0 & L \trn L  \end{pmatrix}
\end{equation*}
that 
\begin{equation*}
\begin{split}
\det \left( \mathscr{S}(L)^2 - \lambda I_{2 \abs{\calX}} \right) = \det \left(  L L \trn - \lambda I_{\abs{\calX}} \right)^2,
\end{split}
\end{equation*}
thus the $\mu_x$ are exactly the $\pm \sqrt{\lambda_x}$. It follows that the $\mu_x$ are also real, and that
$$\gamma_\ddagger(P) = \gamma_\star(\mathscr{S}_{\pi}(P)) = \gamma(\mathscr{S}_{\pi}(P)) = 1 - \sqrt{\lambda^{\downarrow}_2},$$
which proves the equality at $(b)$.
Inequalities $(a)$ and $(c)$ follow from $\gamma_\ddagger(P) \in [0, 1]$.
\end{proof}

This leads to a reformulation of the pseudo-spectral gap, hence an equivalent way of bounding the mixing time of an ergodic chain from above and below. Indeed, defining
\begin{equation}
\label{definition:dilated-pseudo-spectral-gap}
\gamma \dps \eqdef \max_{k \in \N} \set{ \frac{1}{k} \gamma_\ddagger\left(P^k\right)},
\end{equation}
it holds that
\begin{equation}
\label{eq:gamma-dps-and-gamma-ps}
\gamma \dps \leq \gamma \ps = \max_{k \in \N} \set{ \frac{1}{k} \gamma_\ddagger\left(P^k\right)\left(2 - \gamma_\ddagger\left(P^k\right)\right)  }\leq 2 \gamma \dps.
\end{equation}
Combining
with \citet[Proposition~3.4]{paulin2015concentration},
we obtain
\begin{equation}
\label{eq:dilated-pseudo-spectral-gap-controls-mixing-time}
\frac{1}{4\gamma \dps} \leq \tmix \leq \frac{1}{\gamma \dps} \log \frac{4e}{\pimin}.
\end{equation}
\citet[Lemma~15]{pmlr-v99-wolfer19a} established that for irreducible
reversible Markov chains, the absolute spectral gap and the pseudo-spectral gap are within a multiplicative
factor of $2$,
$$
\gamma_\star \leq \gamma \ps = \gamma_\star(2 -  \gamma_\star) \leq 2\gamma_\star.$$
The below-stated Lemma~\ref{lemma:pssg-reversible} shows that
$\gamma \dps$ reduces more naturally to $\gamma_\star$ under reversibility.
\begin{lemma}
  \label{lemma:pssg-reversible}
  For ergodic and reversible
  $P$, it holds that
  $\gamma \dps = \gamma_\star$.
\end{lemma}
\begin{proof}
By reversibility,
  $P^\star = P$, $P$ has a real spectrum, and $\mathscr{S}_\pi(P)^2 = \begin{psmallmatrix}P^2 & 0 \\ 0 & P^2\end{psmallmatrix}$.
  We denote by $1 = \lambda^{\downarrow}_1 > \lambda^{\downarrow}_2 \geq \dots \geq \lambda^{\downarrow}_{\abs{\calX}}$ the eigenvalues of $P$. For any $x \in \calX$ and $k \geq 1$, $\lambda_x^{\downarrow 2k} \geq 0$ is an eigenvalue for $P^{2k}$.
  From the proof of Lemma~\ref{lemma:multiplicative-and-dilation-similar-spectrum}, it follows that the eigenvalues of $\mathscr{S}_\pi(P^k)$ are the $\pm \sqrt{ \lambda_x^{\downarrow 2k}} = \pm \abs{\lambda_x}^k$.
  Furthermore $\lambda_\star^{k}$ with $\lambda_\star = \max \set{\lambda^{\downarrow}_2, \abs{\lambda^{\downarrow}_{\abs{\calX}}}}$ is necessarily the second largest. As a result,
\begin{equation*}
\gamma \dps = \max_{k \in \N} \left\{ \frac{1 - \lambda_\star^{k}}{k} \right\} = \gamma_\star.
\end{equation*}
\end{proof}

\subsection{Improved empirical estimator and confidence intervals}
\label{section:empirical-confidence-intervals}
We estimate the quantity $\gamma \dps$ as a proxy for $\gamma \ps$.
Namely, for fixed $K \in \N$, we define
\begin{equation*}
\begin{split}
\label{equation:definition-overall-plug-in-estimator}
\widehat{\gamma} \dpsK \eqdef \max_{k \in [K]} \set{ \frac{1}{k} \gamma_\ddagger\left( \widehat{P}^{(k, \alpha)} \right) },
\end{split}
\end{equation*}
where the $\alpha$-smoothed tally variables, with $\alpha >0$, are defined by
\begin{equation*}
\begin{split}
\widehat{P}^{(k, \alpha)} \eqdef \sum_{x,x' \in \calX} \frac{N_{x x'}^{(k)} + \alpha}{N_x^{(k)} + \abs{\calX} \alpha} e_x \trn e_{x'}, \qquad \widehat{\pi}^{(k, \alpha)} \eqdef \sum_{x \in \calX} \frac{N_{x}^{(k)} + \abs{\calX} \alpha}{\floor{(m - 1) / k} + \abs{\calX}^2 \alpha} e_x, \\
\widehat{D}_\pi^{(k, \alpha)} \eqdef \diag\left( \widehat{\pi}^{(k, \alpha)} \right), \qquad \widehat{L}^{(k, \alpha)} \eqdef \left( \widehat{D}_\pi^{(k, \alpha)} \right)^{1/2}  \widehat{P}^{(k, \alpha)} \left( \widehat{D}_\pi^{(k, \alpha)} \right)^{-1/2},
\end{split}
\end{equation*}
and where $N_x^{(k)}, N_{x x'}^{(k)}$ are introduced in \eqref{equation:definition-natural-counting-random-variables}. Notice that $\widehat{\pi}^{(k, \alpha)}$ requires more aggressive smoothing than for the transition matrix in order to ensure stationarity, and has an easily computable form, which is an improvement over \citet{hsu2019}, where the empirical stationary distribution is computed by solving a linear system. 

\begin{remark}[Computational complexity.]
Computing the estimator $\widehat{\gamma} \dpsK$ can be achieved in $\tilde \bigO(K(m + \abs{\calX}^2))$ operations (see discussion in Section~\ref{section:computational-complexity}).
\end{remark}

We make a data-driven choice for $K = \widehat{K}(X_1, \dots, X_m)$ and adapt the proof of \citet[Theorem~8]{pmlr-v99-wolfer19a} to recover fully empirical confidence intervals for $\gamma \dps$ instead of $\gamma \ps$, whose
non-asymptotic form is described in the next theorem.

\begin{theorem}
  \label{theorem:confidence-intervals}
Let $c \leq 48$ be a universal constant.
With probability at least $1 - \Theta(\delta)$ it holds that
\begin{equation*}
\begin{split}
\abs{\widehat{\gamma} \dpshatK - \gamma \dps} &\leq \frac{1}{\widehat{K}} + \max_{k \in [\widehat{K}]} \set{\frac{1}{k} \left( \widehat{V}_{\widehat{\delta} }^{(k, \alpha)}  + \widehat{U}_{\widehat{\delta}}^{(k, \alpha)}\left(2 + \widehat{U}_{\widehat{\delta}}^{(k, \alpha)} \right) \right)},  \\
\end{split}
\end{equation*}
where
\begin{equation}
\begin{split}
\label{eq:adaptive-prefix-empirical}
    \widehat{\delta} =  \sqrt{\frac{\log^3 m}{m }}\frac{\delta}{\widehat{K} \abs{\calX}}, \qquad 
    \widehat{K} = \ceil*{\frac{\Nmin^{3/2}}{m \log^{3/2} m}}, \\
\end{split}
\end{equation}
and
\begin{equation*}
\begin{split}
    \widehat{T}_\delta^{(k, \alpha)} &= \cfrac{c}{\gamma \ps\left(\widehat{P}^{(k, \alpha)}\right)} \log \left( 2 \sqrt{\cfrac{2 (\floor{(m-1) / k} + \alpha \abs{\calX}^2 )}{\Nmin^{(k)} + \alpha \abs{\calX}  }} \right) \widehat{W}_\delta^{(k, \alpha)} \\
    \widehat{U}_\delta^{(k, \alpha)} &= \frac{1}{2} \max \bigcup_{x \in \calX} \set{\cfrac{\widehat{T}_\delta^{(k,\alpha)}}{\cfrac{N_x^{(k)} + \alpha\abs{\calX} }{\floor{(m-1)/k} + \alpha \abs{\calX}^2 }}, \cfrac{\widehat{T}_\delta^{(k,\alpha)}}{\left[\cfrac{N_x^{(k)} + \alpha \abs{\calX} }{\floor{(m-1)/k} + \alpha \abs{\calX}^2 } - \widehat{T}_\delta^{(k,\alpha)}\right]_+}} \\
		\widehat{V}_\delta^{(k,\alpha)} &= \sqrt{\abs{\calX}}  \cfrac{\Nmax^{(k)} + \alpha \abs{\calX} }{\Nmin^{(k)} + \alpha \abs{\calX} } \widehat{W}_\delta^{(k, \alpha)}  \\
		\widehat{W}_\delta^{(k,\alpha)} &= 2 \max_{x \in \calX} \set{ \cfrac{  \sum_{x' \in \calX} \sqrt{N^{(k)}_{xx'}} + 3 \sqrt{N^{(k)}_x / 2} \log^{1/2} {(2 \floor{(m-1)/k} \abs{\calX} /\delta)}  + \alpha \abs{\calX} }{N^{(k)}_x + \alpha \abs{\calX} }}. \\
\end{split}
\end{equation*}
\end{theorem}

\begin{proof}
We first prove the theorem for a fixed $K \in \N$.
\begin{equation*}
\begin{split}
\abs{\widehat{\gamma} \dpsK - \gamma \dps} &\leq \frac{1}{K} + \abs{\widehat{\gamma} \dpsK - \gamma \dpsK} \\
&\leq \frac{1}{K} + \max_{k \in [K]} \set{  \frac{1}{k} \abs{ \gamma_\star \left( \mathscr{S}\left(\widehat{L}^{(k)}\right)\right) - \gamma_\star \left( \mathscr{S}\left(L^{k}\right)\right) } } \\
&\leq \frac{1}{K} + \max_{k \in [K]} \set{  \frac{1}{k}  \nrm{  \mathscr{S}\left(\widehat{L}^{(k)}\right) - \mathscr{S}\left(L^{k}\right)}}\\
&\stackrel{(\star)}{=} \frac{1}{K} + \max_{k \in [K]} \set{  \frac{1}{k}  \nrm{ \widehat{L}^{(k)} - L^{k}}},
\end{split}
\end{equation*}
where for $(\star)$, it suffices to notice that for any real matrix $A$, $\mathscr{S}(A) \trn \mathscr{S}(A)$
is block diagonal.
The next step is similar to the proof of  \citet{pmlr-v99-wolfer19a}, where we further bound the right-hand side using
\begin{equation*}
    \nrm{ \widehat{L}^{(k)} - L^{k}} \leq \nrm{\calE_P^{(k, \alpha)}} + \nrm{\calE_\pi^{(k, \alpha)}}\left(2 + \nrm{\calE_\pi^{(k, \alpha)}}\right),
\end{equation*}
where
\begin{equation*}
    \nrm{\calE_\pi^{(k, \alpha)}} = \max \set{ \nrm{\calE_{\pi,1}^{(k, \alpha)}}, \nrm{\calE_{\pi,2}^{(k, \alpha)}} },
\end{equation*}
with
\begin{equation*}
    \begin{split}
    \calE_P^{(k, \alpha)} &= \left(\widehat{D}_\pi^{(k, \alpha)}\right)^{1/2}\left( \widehat{P}^{(k, \alpha)} - P^k \right) \left(D_\pi^{(k, \alpha)}\right)^{-1/2}, \\
        \calE_{\pi,1}^{(k, \alpha)} &= \left(\widehat{D}_\pi^{(k, \alpha)}\right)^{1/2} \left(D_\pi^{(k, \alpha)}\right)^{-1/2} - I, \\
        \calE_{\pi,2}^{(k, \alpha)} &= \left(D_\pi^{(k, \alpha)}\right)^{1/2} \left(\widehat{D}_\pi^{(k, \alpha)}\right)^{-1/2} - I. \\
    \end{split}
\end{equation*}
By sub-multiplicativity of the spectral norm and H\"older's inequality,
\begin{equation*}
    \nrm{\calE_P^{(k, \alpha)}} \leq \sqrt{\abs{\cal{X}} } \frac{\Nmax^{(k)} + \alpha \abs{\calX} }{\Nmin^{(k)}  + \alpha \abs{\calX} } \nrm{\widehat{P}^{(k)} - P^k}_\infty.
\end{equation*}
Instead of relying on a matrix Freedman inequality as in \citet{pmlr-v99-wolfer19a}, 
we use a technique developed in \citet{wolfer2022empirical}.
The method applies as a black-box 
the distribution learning result of \citet[Theorem~2.1]{cohen2020learning}, which states that
for an independent vector $X_1, \dots, X_n$ sampled from an unknown distribution $\mu \in \calP(\calX)$, it holds that
with probability at least $1 - \delta$,
\begin{equation*}
    \tv{\mu - \widehat{\mu}(X_1, \dots, X_n)} \leq \frac{1}{\sqrt{n}} \sum_{x \in \calX} \sqrt{\widehat{\mu}(x)} + 3 \sqrt{\frac{\log 2/\delta}{2n}},
\end{equation*}
where $\widehat{\mu}$ is the empirical distribution computed from the sample.
This enables us to establish an error bound for learning the conditional distribution defined by each state.
Indeed, an independent sample for the conditional distribution is obtained by the simulation argument of \citet[p.19]{billingsley1961statistical}.
We obtain \citep[Theorem~3.1]{wolfer2022empirical} that with probability at least $1 - \delta$,
\begin{equation*}
\begin{split}
\nrm{\widehat{P}^{(k)} - P^k}_\infty &\leq \widehat{W}_{\delta}^{(k, \alpha)},\\
\end{split}
\end{equation*}
with
\begin{equation*}
\begin{split}
\widehat{W}_{\delta}^{(k, \alpha)} &= 2 \max_{x \in \calX} \set{ \frac{  \sum_{x' \in \calX} \sqrt{N^{(k)}_{xx'}} + (3/\sqrt{2}) \sqrt{N^{(k)}_x} \sqrt{\log {(2 \floor{(m-1)/k} \abs{\calX} /\delta)}}  + \alpha \abs{\calX} }{N^{(k)}_x + \alpha \abs{\calX} }}.
\end{split}
\end{equation*}
The term $\|\calE_\pi^{(k, \alpha)}\|$ is controlled by $\widehat{V}_\delta^{(k, \alpha)}$ similar to \citet{pmlr-v99-wolfer19a}.
A union bound completes the proof for adaptive $\widehat{K}$ by noting that $\widehat{K} \in \N$ and $\widehat{K} \leq \sqrt{m / \log^3 m} + 1$.
\end{proof}

\begin{remark}
The algorithm will only start to consider spectral gaps of $k$-skipped chains with $k > 1$ once the least visited state was seen a significant amount of times.
\end{remark}

\begin{remark}
Note that unlike the estimator, computing the confidence intervals still requires matrix multiplication.
\end{remark}

\subsubsection{Asymptotic behavior}
\citet[Theorem~8]{pmlr-v99-wolfer19a} involved looking at skipped chains with different offsets. In Theorem~\ref{theorem:confidence-intervals} however, we only consider a null offset as the confidence intervals follow the same asymptotic behavior, at reduced computational cost.
Let use denote 
$$\Gamma(P) = \max_{x \in \calX} \set{ \frac{\nrm{e_x P}_{1/2}}{\pi(x)}},$$
which depends on $P$ in a  substantially finer way than its worst-case upper bound $\frac{\abs{\calX}}{\pimin}$.
From \citet[Lemma~3.1]{wolfer2022empirical}, the asymptotic behavior of $\widehat{W}_\delta$ is roughly given by
\begin{equation*}
    \widehat{W}_\delta^{(1, \alpha)} \asymp \sqrt{\frac{\log (m \abs{\calX} / \delta)}{m}}\frac{1}{\sqrt{\pimin}} + \sqrt{\frac{\Gamma(P)}{m}}.
\end{equation*}
We observe that unlike intervals in \citet{pmlr-v99-wolfer19a}, the second dominant term in the empirical confidence bound can then largely benefit from sparsity properties of $P$.
Only accounting for the dominant terms,
\begin{equation*}
\begin{split}
\widehat{W}_\delta^{(k, \alpha)} &\asymp \sqrt{ \frac{\log m}{m}} \sqrt{\frac{k}{\pimin}}, \\  
\widehat{V}_\delta^{(k, \alpha)} &\asymp \sqrt{ \frac{\log m}{m}} \frac{\sqrt{k \abs{\calX}}}{\pimin^{3/2}}, \\ 
\widehat{T}_\delta^{(k, \alpha)} &\asymp \sqrt{ \frac{\log^3 m}{m}}\frac{\log 1 / \pimin}{\gamma \ps \sqrt{\pimin} \sqrt{k}}, \\
\widehat{U}_\delta^{(k, \alpha)} &\asymp \sqrt{ \frac{\log^3 m}{m}}\frac{\log 1 / \pimin}{\gamma \ps \pimin^{3/2} \sqrt{k}}, \\
\end{split}
\end{equation*}
thus for fixed $K \in \N$,
\begin{equation*}
\begin{split}
    \abs{\widehat{\gamma} \dpsK - \gamma \dps} - \frac{1}{K}
    &\asymp \sqrt{ \frac{\log^3 m}{m}}\frac{\log 1 / \pimin}{\gamma \ps \pimin^{3/2}}.
\end{split}
\end{equation*}

\section{Algorithm}
\label{section:algorithm}
The implementation details of the procedure of Section~\ref{section:empirical-estimation} are described in this section. 

\subsection{Pseudo-code}

For clarity, the computation of the confidence intervals is not made explicit in the 
pseudo-code and the reader is referred to Theorem~\ref{theorem:confidence-intervals} 
for their expression. The estimator is based on an approximate plug-in approach of 
$\gamma \dps$ \eqref{definition:dilated-pseudo-spectral-gap}. 
Namely, we compute an approximation of the pseudo-spectral gap define over 
a data-driven prefix $[\widehat{K}] \subsetneq \N$
with respect to reversible dilations of powers of the chain, which are each estimated with the natural counts based on observed 
skipped chains. We additionally introduce a smoothing parameter $\alpha$, 
which for simplicity is kept
fixed
for each power, and whose purpose is to keep 
the confidence intervals and estimator properly defined even in degenerate cases. 
The sub-procedure that computes the spectral gap of the reversiblizations of a chain
invokes
the Lanczos method by computing the first few largest eigenvalues in magnitude,
and is discussed in more details at Remark~\ref{remark:why-lanczos}. 

\begin{algorithm}[H]
 \SetAlgoFuncName{PseudoSpectralGapDilEstimator}{pssg}
 \SetKwData{CI}{$CI$}
 \SetKwData{N}{$\mathbf{N}$}
 \SetKwData{estimator}{$g_\star$}
 \SetKwFunction{SpectralGapRevDil}{SpectralGapRevDil}
 \SetKwProg{Fn}{Function}{:}{}
 \SetKwFunction{FPSSG}{PseudoSpectralGapDil}
	
	\SetKwFunction{DFPSSG}{AdaptivePseudoSpectralGapDil}
	\SetKwData{NN}{$\mathbf{N}$}
	\Fn{\DFPSSG{$\abs{\calX}$, $\alpha$, $(X_1, \dots, X_m)$}}{
	    $\NN \leftarrow \left[ 0 \right]_{\abs{\calX}}$
	    
	    \For{$t \leftarrow 1$ \KwTo $n - 1$}{
						$\NN[X_t] \leftarrow \NN[X_t] + 1$ \\
				}
				
	    $\Nmin = \min \set{ \NN }$
	    
	    $K \leftarrow \ceil*{\Nmin^{3/2}/(m \log^{3/2} m)}$ \\
		\KwRet \FPSSG{$\abs{\calX}$, $\alpha$, $(X_1, \dots, X_m)$, K}
	}
	
	\Fn{\FPSSG{$\abs{\calX}$, $\alpha$, $(X_1, \dots, X_m)$, $K$}}{
		$\estimator \leftarrow 0$ \\
		\For{$k \leftarrow 1$ \KwTo $K$}{
			$g \leftarrow \SpectralGapRevDil(\abs{\calX}, \alpha, (X_{1}, X_{1 + k}, X_{1 + 2k}, \dots, X_{1 + \floor{(m-1)/k}k}))$ \\
			\If{$g / k > \estimator $}{
					$\estimator \leftarrow g / k$
			}
	 }
	 \KwRet \estimator
	}
	
	\SetKwFunction{FSGMR}{SpectralGapRevDil}
	\Fn{\FSGMR{$\abs{\calX}, \alpha, (X_1, \dots, X_n)$}}{
				 \SetKwData{N}{$\mathbf{N}$}
				 \SetKwData{estimator}{$\widehat{\sg}^\mathscr{S}_{k, \alpha}$}
				 \SetKwData{Q}{$\mathbf{Q}$}
				 \SetKwData{NT}{$\mathbf{T}$}
				 \SetKwData{NN}{$\mathbf{N}$}
				 \SetKwData{NS}{$\mathbf{S}$}
				 \SetKwData{S}{$\mathbf{S}$}
				 \SetKwData{L}{$\mathbf{L}$}
                  \SetKwData{I}{$\mathbf{I}$}
				 \SetKwData{B}{$\mathbf{B}$}
				 \SetKwData{D}{$\mathbf{D}$}
				 \SetKwData{PI}{$\mathbf{\Pi}$}
				 \SetKwData{zero}{$\mathbf{0}$}
				 \SetKwFunction{LanczosSecondEigenvalue}{LanczosSecondEigenvalue}
				 \SetKwFunction{Dilate}{Dilate}
				 \SetKwFunction{Diag}{Diag}
				 \SetKwFunction{Concatenate}{Concatenate}
        $\NN \leftarrow \left[ \abs{\calX} \alpha \right]_{\abs{\calX}}$ \\
				$\NT \leftarrow \left[ \alpha \right]_{\abs{\calX} \times \abs{\calX}}$ \\
				\For{$t \leftarrow 1$ \KwTo $n - 1$}{
						$\NN[X_t] \leftarrow \NN[X_t] + 1$ \\
						$\NT[X_t, X_{t+1}] \leftarrow \NT[X_t, X_{t+1}] + 1$ \\
				}
				$\D \leftarrow \Diag(\NN)^{-1/2}$ \\
				$\L \leftarrow \D \NT \D$ \\
				$\S \leftarrow \begin{bmatrix} \zero & \L \\ \L \trn & \zero \end{bmatrix}_{2\abs{\calX} \times 2\abs{\calX}} $ \\
				\KwRet $2 - \LanczosSecondEigenvalue(\S + \I)$
  }
	
\caption{The estimation procedure outputting $\widehat{\gamma} \dpshatK$}
\label{algorithm:pseudo-spectral-gap-estimator}
\end{algorithm}

\begin{remark}
\label{remark:why-lanczos}
From \eqref{equation:similarity-dilation-rescaled}, for any $k \in \N$, $\mathscr{S}(L^k)$ and $\mathscr{S}_{\pi}(P^k)$ are similar matrices, and we can immediately rewrite
$$\gamma \dps = \max_{k \in \N} \set{ \frac{1}{k} \gamma \left( \mathscr{S}(L^k) \right)}.$$
The resulting symmetry of the collection of matrices, as well as the requirement for computing large eigenvalues in magnitude, naturally invites Lanczos-type algorithms.
In fact, from the prior knowledge of the two leading eigenvalues 
($\pm 1$) 
we can further simplify the computation by analytical considerations, and rewrite
$$\gamma \dps = \max_{k \in \N} \set{ \frac{1}{k} \rho\left( \mathscr{S}\left(L^k\right) - \mathscr{S}\left(\sqrt{\pi} \trn \sqrt{\pi}\right) \right) }. $$
\end{remark}

\subsection{Note on the computational complexity}
\label{section:computational-complexity}

Using $K$ dilations, since we have a closed form for the empirical stationary distribution, we manage to keep a quadratic time complexity of $\bigO \left( K \left(m + \abs{\calX}^2  + \mathcal{C}_{\lambda_\star} \right) \right)$,
 where $\mathcal{C}_{\lambda_\star}$ is the complexity of computing the second largest-magnitude eigenvalue of a symmetric matrix.
For this task, we consider
here the Lanczos algorithm.
Let $\lambda_1, \ldots, \lambda_{\abs{\calX}}$ be the eigenvalues of a symmetric real matrix ordered by magnitude,
and
denote by
$\tilde{\lambda}_1$
the algorithm's approximation for $\lambda_1$.
Then,
for a
stochastic
matrix, it is known \citep{kaniel1966estimates, paige1971computation, saad1980rates}
that
\begin{equation*}
\begin{split}
\abs{\lambda_1 - \tilde{\lambda}_1} \leq c R^{-2(n-1)},
\end{split}
\end{equation*}
where $c$ a universal constant,
$n$ is the number of iterations
(in practice often $n \ll \abs{\calX}$), and
$R = 1 + 2 r + 2 \sqrt{r^2 + r}$,
with $r = \frac{\lambda_1 - \lambda_2}{\lambda_2 - \lambda_{\abs{\calX}}}$.
In order to attain additive accuracy
$\eta$,
it therefore suffices to iterate the method
$n \geq 1 + \frac{1}{2} \frac{ \log \left( c \eta^{-1} \right)}{\log (R)} = \bigO \left( \frac{\log \left( \eta^{-1} \right)}{\log (R)} \right)$ times.
A single
iteration
involves multiplying a vector by a matrix,
incurring a cost of
$\bigO \left( \abs{\calX}^2 \right)$,
and so the full complexity of the
Lanczos algorithm
is
$\bigO\left( m + \abs{\calX}^2 \frac{\log \left( \eta^{-1} \right)}{\log (R)} \right)$.
More refined complexity analyses may be found in
\citet{kuczynski1992estimating, arora2005fast}. For comparison, the previous approach of \citet{pmlr-v99-wolfer19a} involved computing $K$ multiplicative reversiblizations, which each requires $\bigO(\abs{\calX}^\omega)$, where $2 \leq \omega \leq 2.3728596$ 
is the best current time complexity of multiplying (and inverting, and diagonalizing) $\abs{\calX} \times \abs{\calX}$ matrices \citep{alman2021refined}.
Our proposed computational method is therefore faster over a non-trivial regime.

\section{Proofs}
\label{section:proofs}
\subsection{Auxiliary lemmas}

\subsubsection{Proof of Lemma~\ref{lemma:technical-correspondence}}
We write
\begin{equation*}
    P - \Pi = \Dpi^{-1/2}\left(L - \sqrt{\pi} \trn \sqrt{\pi}\right) \Dpi^{1/2},
\end{equation*}
and observe that $\sqrt{\pi}$ is an eigenvector of $L$ for eigenvalue $1$.
We verify that $L$ is symmetric, thus we can decompose it as
\begin{equation*}
    L = \sum_{x \in \calX} \lambda_x f_x \trn f_x,
\end{equation*}
with $f_x \in \R^{\calX}$ for any $x \in \calX$, and
where by matrix similarity, $\set{\lambda_x}_{x \in \calX} = \sigma(L) = \sigma(P)$ is real.
Since $P$ is ergodic, there is a unique unit eigenvalue $\lambda_1 = 1$.
Removing the subspace $f_1 \trn f_1 = \sqrt{\pi} \trn \sqrt{\pi}$, it follows that 
$$\rho(P - \Pi) = \rho\left( L - \sqrt{\pi} \trn \sqrt{\pi} \right) = \max \set{\lambda_x \in \sigma(L) \colon \abs{\lambda_x} \neq 1} = 1 - \gamma_\star(P).$$
It follows from symmetry that
\begin{equation*}
    \begin{split}
        \rho\left( L - \sqrt{\pi} \trn \sqrt{\pi} \right) &= \nrm{ L - \sqrt{\pi} \trn \sqrt{\pi} } = \max_{f \neq 0} \frac{\nrm{\left( L - \sqrt{\pi} \trn \sqrt{\pi}\right)f}}{\nrm{f}} \\
        &= \max_{f \neq 0} \frac{\nrm{( P - \Pi)(f/\sqrt{\pi})}_\pi}{\nrm{(f/\sqrt{\pi})}_\pi}, \\
    \end{split}
\end{equation*}
where the last equality is obtained by verifying that for $x \in \calX$,
\begin{equation*}
\begin{split}
    \frac{1}{\sqrt{\pi(x)}}\left(L - \sqrt{\pi} \trn \sqrt{\pi}\right)f(x) &= \frac{1}{\sqrt{\pi(x)}}\sum_{x' \in \calX}\left(L(x,x') - \sqrt{\pi(x)} \sqrt{\pi(x')}\right)f(x') \\
    &= \sum_{x' \in \calX}\left(P(x,x') - \Pi(x,x') \right) \frac{f(x')}{\sqrt{\pi(x')}} = (P - \Pi)\left(f/\sqrt{\pi}\right)(x).
\end{split}
\end{equation*}
By irreducibility of $P$, the mapping $f \rightarrow f / \sqrt{\pi}$ is one-to-one and $f / \sqrt{\pi} = 0$ iff $f = 0$, thus,
\begin{equation*}
    \begin{split}
        \rho\left( L - \sqrt{\pi} \trn \sqrt{\pi} \right) &= \max_{f \neq 0} \frac{\nrm{( P - \Pi)f}_\pi}{\nrm{f}_\pi} = \nrm{P - \Pi}_\pi. \\
    \end{split}
\end{equation*}
\qedsymbol

\subsubsection{Statement and proof of Lemma~\ref{lemma:taylor-type-inequality}}

\begin{lemma}
\label{lemma:taylor-type-inequality}
For any $x \in [0, 1]$ and $p \in \N$,
\begin{equation*}
    1 - (1 - x)^p \geq px (1 - p x/2).
\end{equation*}
\end{lemma}
\begin{proof}
The case $x = 1$ is immediate, and we treat the remaining scenario where $x \neq 1$.
    Write $f_p(x) = 1 - px + p^2 x^2/2 - (1 - x)^p$.
    For a fixed $p \in \N$, we compute the first and second order derivatives
    \begin{equation*}
        \begin{split}
            f'_p(x) &= p \left( (1 - x)^{p-1} + px - 1 \right), \\
            f''_p(x) &= p^2 \left( 1 - \left( \frac{p-1}{p}\right)(1 - x)^{p - 2} \right).
        \end{split}
    \end{equation*}
    For $p = 1$, $f''_p(x) = 1 > 0$, and for $p \geq 2$, we have $(p-1)/p < 1$ and $(1 - x)^{p - 2} \leq 1$, thus $f''_p(x) > 0$.
    It follows that $f'_p$ is strictly increasing on $[0,1)$ and since $f'_p(0) = 0$, we also have that $f_p$ is non-decreasing on $[0,1]$. Finally, $f_p(0) = 0$ finishes proving the inequality.
\end{proof}

\subsection{Proof of Lemma~\ref{lemma:learn-laplacian-spectral-norm}}
\label{section:proofs-laplacian}
The proof follows a similar strategy to that of \citet{hsu2019}. We highlight the differences when they occur and streamline parts of the argument. We first assume that the chain is started stationarily, and will accommodate for the non-stationary case at the end of the proof.

\subsubsection{Control the probability of an ill-defined estimator}

From an application of the Bernstein inequality of \citet[Theorem~3.11]{paulin2015concentration}, it is sufficient to have
$m \geq \frac{c}{\pi_\star \gamma \ps} \log \frac{1}{\pi_\star \delta}$, with $c > 0$ a natural constant, in order to ensure that for any $x \in \calX$, $N_x > 0$ with probability at least $1 - \delta/4$.

\subsubsection{Reduction to decoupled blocks}
On the event where our estimators are well-defined, the following decomposition holds,
\begin{equation*}
\begin{split}
&\widehat{L} - L = \EE{Q} + \EE{\pi,1} \widehat{L}  + \widehat{L} \EE{\pi,2} - \EE{\pi,1} \widehat{L} \EE{\pi,2},
\end{split}
\end{equation*}
where
\begin{equation*}
\begin{split}
\EE{Q} \doteq \Dpi^{-1/2} \left( \widehat{Q} - Q \right) \Dpi^{-1/2}, \qquad \EE{\pi,1} \eqdef I - \Dpi^{1/2} \widehat{D}_{\pi}^{-1/2}, \qquad \EE{\pi,2} \eqdef I - \widehat{D}_{\pi}^{1/2}\Dpi^{-1/2}, \\
\end{split}
\end{equation*}
and which can be instructively compared with the decomposition in \citet[Lemma~6.2]{hsu2019}.
From the Perron-Frobenius theorem, $\| \widehat{L} \| = 1$, as $\sqrt{\widehat{\pi}}$ is an eigenvector associated to
eigenvalue $1$ for $\widehat{L} \trn \widehat{L}$. This allows us to upper bound the spectral norm,
\begin{equation}
\label{eq:control-spectral-norm}
\begin{split}
\nrm{\widehat{L} - L} &\leq \nrm{\EE{Q}} + \nrm{\EE{\pi}} \left( 2  + \nrm{\EE{\pi}} \right), 
\end{split}
\end{equation}
where
we use the shorthand
$\nrm{\EE{\pi}} \eqdef \max \set{\nrm{\EE{\pi,1}}, \nrm{\EE{\pi,2}}}$.
From Lemma~\ref{lemma:estimation-pimin-relative}, and the technique described in \citet[Section~6.3]{hsu2019},
for $m \geq  \frac{c}{\pimin \gamma \ps \eps^2} \log \frac{1}{\pimin \delta}$, with probability at least $1 - \delta/4$,
\begin{equation}
\label{eq:control-diagonal-stationary-terms}
\nrm{\EE{\pi}} \leq \eps/6
\end{equation}
for some universal $c > 0$.
It remains to control the spectral norm of $\EE{Q}$. 
From the simple observation that $\E[\pi]{\EE{Q}} = 0$, 
we are left with the task of bounding the fluctuations of the matrix $\EE{Q}$ 
around its mean in spectral norm, when the chain is started from its stationary distribution. 
Suppose for simplicity of the analysis that $m = 2 B s + 1$ for 
\begin{equation}
\label{eq:block-size}
    s = \ceil*{ \frac{4}{\gamma \ps} \log \frac{m \sqrt{e} }{2 \sqrt{\pimin} \delta}},
\end{equation}
and some $B \in \N$ (a simple argument extends the proof beyond this case). 
We partition the trajectory $X_1, \dots, X_m$ into $2B$ blocks of size $s$ , 
$$X^{[1]}, X^{[2]},\dots,X^{[2B]}.$$
Denote for $b \in [B]$,
\begin{equation*}
\begin{split}
\widehat{Q}^{[2b - 1]} &\eqdef \frac{1}{m-1} \sum_{t = (2b - 2)s + 1}^{(2b - 1)s}  {e_{X_t}} \trn e_{X_{t+1}}, \\
\widehat{Q}^{[2b]} &\eqdef \frac{1}{m-1} \sum_{t = (2b - 1)s + 1}^{2b s} {e_{X_t}} \trn e_{X_{t+1}}, \\
\end{split}
\end{equation*}
and we further decompose $\EE{Q}$ as,
\begin{equation}
\label{eq:decomposition-into-two-series-of-blocks}
\begin{split}
\EE{Q} &= \sum_{b = 1}^{B} \EE{Q}^{[2b]} + \sum_{b = 1}^{B} \EE{Q}^{[2b - 1]}\\
\end{split}
\end{equation}
with 
$$\EE{Q}^{[2b]} \doteq \Dpi^{-1/2} \left( \widehat{Q}^{[2b]} - \frac{1}{2B}Q \right) \Dpi^{-1/2},$$ 
and  $\EE{Q}^{[2b - 1]}$ is similarly defined. 
From a union bound argument, sub-additivity of the spectral norm and symmetry of expressions we shall focus on the deviation for the first summand in \eqref{eq:decomposition-into-two-series-of-blocks}. 
From \citet[Corollary~2.7]{yu1994rates},
\begin{equation*}
\begin{split}
\PR[\pi]{\nrm{\sum_{b = 1}^{B} \EE{Q}^{[2b]}} > \eps/4} \leq \PR[\pi]{\nrm{\sum_{b = 1}^{B} \tilde{\mathcal{E}}_{Q}^{[2b]}} > \eps/4} + (B - 1)\beta(s) \\
\end{split}
\end{equation*}
where each $\tilde{\mathcal{E}}_{Q}^{[2b]}$ is computed assuming the chain has 
restarted from its stationary distributions, and where
$\beta(s)$ is the beta coefficient \citep{bradley2005basic} of the process.
For a homogeneous Markov chain, the $\beta$-coefficients can be expressed \citep{mcdonald2011estimating} as
$$\beta(s) = \sum_{x \in \calX} \pi(x) \tv{e_x P^{s} - \pi},$$
and from \citet[(3.10)]{paulin2015concentration}, for any $x \in \calX$,
\begin{equation*}
\begin{split}
    \tv{e_x P^s - \pi}  \leq \frac{1}{2}(1 - \gamma \ps)^{(s - 1/\gamma \ps)/2}\sqrt{\nrm{e_x / \pi}_\pi - 1},
\end{split}
\end{equation*}
hence by definition of $B$, and since $\nrm{e_x/\pi}_\pi \leq 1/\pimin$, we obtain
\begin{equation*}
    (B - 1)\beta(s) \leq \frac{m \sqrt{e} }{4s \sqrt{\pimin}} e^{- s \gamma \ps / 2},
\end{equation*}
For our choice of $s$ in \eqref{eq:block-size}, we have in particular $s > 2/\gamma \ps$, and since $e^{-x}/x \leq e^{-x/2}$ for $x \in (1, \infty)$ and $\gamma \ps \leq 1$,
\begin{equation*}
\begin{split}
    (B - 1)\beta(s) \leq \frac{m \sqrt{e} }{8 \sqrt{\pimin}}  e^{- s \gamma \ps / 4},
\end{split}
\end{equation*}
which is smaller than $\delta/4$ by plugging-in \eqref{eq:block-size}.

\subsubsection{Application of a matrix Bernstein inequality}

Following \citet{hsu2019}, we will rely on the following concentration inequality.

\begin{theorem}[Matrix Bernstein inequality {\citep[Theorem~6.1.1]{tropp2015introduction}}]
\label{theorem:matrix-bernstein}
Let $Z_1, \dots, Z_n$ be a sequence of independent random $\abs{\calX} \times \abs{\calX}$ matrices.
Suppose that for any $t \in [n]$, $\E{Z_t} = 0$, and $\|Z_t\| \leq R$.
Let $S = \sum_{t = 1}^{n} Z_t$ and write
$$\Sigma^2 \eqdef \max \set{ \nrm{\E{\sum_{t=1}^{n} Z_t Z_t \trn}}, \nrm{\E{\sum_{t=1}^{n} Z_t \trn Z_t}} }.$$
Then for any $\eps > 0$,
$$\PR{\nrm{S} \geq \eps} \leq 2\abs{\calX} \exp \left( -\frac{\eps^2 / 2}{\Sigma^2 + R \eps/3} \right).$$
\end{theorem}

From sub-additivity of the spectral norm,
\begin{equation*}
\begin{split}
  \nrm{\EE{Q}^{[2b]}} &\leq \frac{1}{\pimin} \nrm{Q^{[2b]}} + \frac{1}{2B} \nrm{L} \leq \frac{s}{m - 1}\left( \frac{1}{\pimin} + 1 \right)
  .
\end{split}
\end{equation*}
We now control the variance term for a given $b \in [B]$. A direct computation yields
\begin{equation*}
\begin{split}
  &\E[\pi]{\EE{Q}^{[2b]} \left(\EE{Q}^{[2b]}\right)\trn}  = \E[\pi]{ \Dpi^{-1/2} \widehat{Q}^{[2b]} \Dpi^{-1/2} \left( \Dpi^{-1/2} \widehat{Q}^{[2b]}  \Dpi^{-1/2} \right) \trn } - \frac{1}{(2B)^2} L L \trn
  .
\end{split}
\end{equation*}
Let $u \in \R^{\calX}$ such that $\nrm{u} = 1$, then 
\begin{equation*}
\begin{split}
&(m-1)^2 \nrm{\E[\pi]{ \Dpi^{-1/2} \widehat{Q}^{[2b]} \Dpi^{-1/2} \left( \Dpi^{-1/2} \widehat{Q}^{[2b]}  \Dpi^{-1/2} \right) \trn }} \\
& = (m-1)^2 u \E[\pi]{ \Dpi^{-1/2} \widehat{Q}^{[2b]} \Dpi^{-1/2} \left( \Dpi^{-1/2} \widehat{Q}^{[2b]}  \Dpi^{-1/2} \right) \trn } u \trn \\ 
  & = \sum_{r,s = (2b - 1)s + 1}^{2b s} \E[\pi]{ u  \Dpi^{-1/2} e_{X_r} \trn e_{X_{r+1}}  \Dpi^{-1/2} \Dpi^{-1/2} e_{X_{s+1}} \trn e_{X_{s}}  \Dpi^{-1/2} u \trn}
  .
\end{split}
\end{equation*}
As in \citet{hsu2019}, from Cauchy-Schwarz and AM-GM inequalities,
\begin{equation*}
\begin{split}
&u  \Dpi^{-1/2} e_{X_r} \trn e_{X_{r+1}}  \Dpi^{-1/2} \Dpi^{-1/2} e_{X_{s+1}} \trn e_{X_{s}}  \Dpi^{-1/2} u \trn \\
&\leq \frac{1}{2} u \Dpi^{-1/2}  e_{X_r} \trn e_{X_{r+1}}  \Dpi^{-1/2} \Dpi^{-1/2} e_{X_{r+1}} \trn e_{X_r}  \Dpi^{-1/2} u \trn \\
& + \frac{1}{2} u \Dpi^{-1/2}  e_{X_s} \trn e_{X_{s+1}}  \Dpi^{-1/2} \Dpi^{-1/2} e_{X_{s+1}} \trn e_{X_s}  \Dpi^{-1/2} u \trn. \\
\end{split}
\end{equation*}
We compute,
\begin{equation*}
\begin{split}
&\E[\pi]{   \Dpi^{-1/2} e_{X_t} \trn e_{X_{t+1}}  \Dpi^{-1/2} \Dpi^{-1/2} e_{X_{t+1}} \trn e_{X_{t}}  \Dpi^{-1/2} } \\
&= \E[\pi]{\sum_{x, x, x'' \in \calX} \frac{\pred{X_t = x} \pred{X_{t+1} = x''} \pred{X_{t+1} = x''} \pred{X_t = x'} }{\pi(x'') \sqrt{\pi(x) \pi(x')}} {e_x} \trn e_{x'} } \\
&= \sum_{x \in \calX} \sum_{x' \in \calX} \frac{P(x,x')}{\pi(x')} {e_x} \trn e_x, \\
\end{split}
\end{equation*}
where
\begin{equation*}
\begin{split}
\nrm{\sum_{b = 1}^{B} \E[\pi]{\EE{Q}^{[2b]} \left(\EE{Q}^{[2b]}\right)\trn}} &\leq \frac{s}{2(m-1)} \left(  4 \nrm{ \sum_{x \in \calX} \sum_{x' \in \calX} \frac{P(x,x')}{\pi(x')} {e_x} \trn e_x} + 1 \right). \\
\end{split}
\end{equation*}
By a very similar derivation,
\begin{equation*}
\begin{split}
\nrm{\sum_{b = 1}^{B} \E[\pi]{\left(\EE{Q}^{[2b]}\right)\trn \EE{Q}^{[2b]}}} &\leq \frac{s}{2(m-1)} \left( 4 \nrm{ \sum_{x \in \calX} \sum_{x' \in \calX} \frac{P(x',x)}{\pi(x)} {e_x} \trn e_x} + 1 \right). \\
\end{split}
\end{equation*}
From definition of $\pi$, and in particular regardless of reversibility,
\begin{equation*}
\begin{split}
\sum_{x' \in \calX} \frac{P(x,x')}{\pi(x')} \leq \frac{1}{\pimin}, \qquad \sum_{x' \in \calX} \frac{P(x',x)}{\pi(x)} \leq \frac{1}{\pimin \pi(x)} \sum_{x' \in \calX} \pi(x') P(x',x) = \frac{1}{\pimin}, \\
\end{split}
\end{equation*}
and so,
the spectral norm of this diagonal matrix is at most $\Sigma^2 \leq \frac{s}{2(m-1)} \left( \frac{4}{\pimin} + 1 \right)$.
From Theorem~\ref{theorem:matrix-bernstein}, for $m \geq c  \frac{s}{\pimin \eps^2} \log \frac{\abs{\calX}}{\delta}$,
\begin{equation}
\label{eq:control-Q-end}
\begin{split}
\nrm{\sum_{b = 1}^{B} \tilde{\mathcal{E}}_{Q}^{[2B]}} \leq \eps/4
\end{split}
\end{equation}
with probability $1 - \delta/4$, where $c > 0$ is universal.
Recall that our choice of $s$ depends on $\log m$. For $a, b > 0$, 
$$a = 2b \log b \implies a \geq b \log a,$$
thus solving the inequality for $m$,
it follows that for
\begin{equation*}
    m \geq \frac{c}{\gamma \ps \pimin \eps^2} \log \frac{\abs{\calX}}{\delta} \log \left( \frac{\log \abs{\calX} / \delta}{\gamma \ps \pimin \delta \eps} \right),
\end{equation*}
where $c > 0$ is universal,
a combination of \eqref{eq:control-spectral-norm}, \eqref{eq:control-diagonal-stationary-terms}, 
\eqref{eq:decomposition-into-two-series-of-blocks} and \eqref{eq:control-Q-end} yields the lemma in the stationary setting.

\subsubsection{Reduction to a chain started from its stationary distribution}
From \citet[Proposition~3.15]{paulin2015concentration}, we can quantify the price of a non-stationary start,
\begin{equation*}
\begin{split}
\PR[\mu]{\nrm{\widehat{L} - L} > \eps} \leq \sqrt{\nrm{\mu/\pi}_{\pi}} \PR[\pi]{\nrm{\widehat{L} - L} > \eps} ^{1/2},
\end{split}
\end{equation*}
with $\nrm{\mu/\pi}_{\pi}$ and obtain the final sample complexity from $\nrm{\mu/\pi}_{\pi} \leq 1/\pimin$ and by simplifying logarithms.
\qedsymbol

\subsubsection*{Acknowledgments}
We are thankful to Daniel Paulin for enlightening conversations and to an anonymous referee for pointing out to us the multiplicativity property of $\gamma \ps$.

\subsubsection*{Funding}
GW is supported by the Special Postdoctoral Researcher Program (SPDR) of RIKEN.
AK was partially supported by
the Israel Science Foundation
(grant No. 1602/19), an Amazon Research Award,
and the Ben-Gurion University Data Science Research Center.

\bibliography{bibliography}

\begin{thebibliography}{53}
\providecommand{\natexlab}[1]{#1}
\providecommand{\url}[1]{\texttt{#1}}
\expandafter\ifx\csname urlstyle\endcsname\relax
  \providecommand{\doi}[1]{doi: #1}\else
  \providecommand{\doi}{doi: \begingroup \urlstyle{rm}\Url}\fi

\bibitem[Alman and Williams(2021)]{alman2021refined}
J.~Alman and V.~V. Williams.
\newblock A refined laser method and faster matrix multiplication.
\newblock In \emph{Proceedings of the 2021 ACM-SIAM Symposium on Discrete
  Algorithms (SODA)}, pages 522--539. SIAM, 2021.

\bibitem[Arora et~al.(2005)Arora, Hazan, and Kale]{arora2005fast}
S.~Arora, E.~Hazan, and S.~Kale.
\newblock Fast algorithms for approximate semidefinite programming using the
  multiplicative weights update method.
\newblock In \emph{Foundations of Computer Science, 2005. FOCS 2005. 46th
  Annual IEEE Symposium on}, pages 339--348. IEEE, 2005.

\bibitem[Bierkens(2016)]{bierkens2016non}
J.~Bierkens.
\newblock Non-reversible {M}etropolis-{H}astings.
\newblock \emph{Statistics and Computing}, 26\penalty0 (6):\penalty0
  1213--1228, 2016.

\bibitem[Billingsley(1961)]{billingsley1961statistical}
P.~Billingsley.
\newblock Statistical methods in {M}arkov chains.
\newblock \emph{The Annals of Mathematical Statistics}, pages 12--40, 1961.

\bibitem[Bradley et~al.(2005)]{bradley2005basic}
R.~C. Bradley et~al.
\newblock Basic properties of strong mixing conditions. a survey and some open
  questions.
\newblock \emph{Probability surveys}, 2:\penalty0 107--144, 2005.

\bibitem[Chen et~al.(1999)Chen, Lov{\'a}sz, and Pak]{chen1999lifting}
F.~Chen, L.~Lov{\'a}sz, and I.~Pak.
\newblock Lifting {M}arkov chains to speed up mixing.
\newblock In \emph{Proceedings of the thirty-first annual ACM symposium on
  Theory of computing}, pages 275--281. ACM, 1999.

\bibitem[Chen and Hwang(2013)]{chen2013accelerating}
T.-L. Chen and C.-R. Hwang.
\newblock Accelerating reversible {M}arkov chains.
\newblock \emph{Statistics \& Probability Letters}, 83\penalty0 (9):\penalty0
  1956--1962, 2013.

\bibitem[Choi(2020)]{choi2020metropolis}
M.~C. Choi.
\newblock {M}etropolis--{H}astings reversiblizations of non-reversible {M}arkov
  chains.
\newblock \emph{Stochastic Processes and their Applications}, 130\penalty0
  (2):\penalty0 1041--1073, 2020.

\bibitem[Cohen et~al.(2020)Cohen, Kontorovich, and Wolfer]{cohen2020learning}
D.~Cohen, A.~Kontorovich, and G.~Wolfer.
\newblock Learning discrete distributions with infinite support.
\newblock In \emph{Advances in Neural Information Processing Systems},
  volume~33, pages 3942--3951, 2020.

\bibitem[Combes and Touati(2019)]{combes2019computationally}
R.~Combes and M.~Touati.
\newblock Computationally efficient estimation of the spectral gap of a
  {M}arkov chain.
\newblock \emph{Proceedings of the ACM on Measurement and Analysis of Computing
  Systems}, 3\penalty0 (1):\penalty0 1--21, 2019.

\bibitem[Diaconis et~al.(2000)Diaconis, Holmes, and Neal]{diaconis2000analysis}
P.~Diaconis, S.~Holmes, and R.~M. Neal.
\newblock Analysis of a nonreversible {M}arkov chain sampler.
\newblock \emph{Annals of Applied Probability}, pages 726--752, 2000.

\bibitem[Fill(1991)]{fill1991eigenvalue}
J.~A. Fill.
\newblock Eigenvalue bounds on convergence to stationarity for nonreversible
  {M}arkov chains, with an application to the exclusion process.
\newblock \emph{The annals of applied probability}, pages 62--87, 1991.

\bibitem[Garnier(2021)]{garnier2021machine}
R.~Garnier.
\newblock \emph{Machine Learning sur les s{\'e}ries temporelles et applications
  {\`a} la pr{\'e}vision des ventes pour l'E-Commerce}.
\newblock PhD thesis, CY Cergy Paris Universit{\'e}, 2021.

\bibitem[Garnier et~al.(2022)Garnier, Langhendries, and
  Rynkiewicz]{garnier2022hold}
R.~Garnier, R.~Langhendries, and J.~Rynkiewicz.
\newblock Hold-out estimates of prediction models for {M}arkov processes.
\newblock \emph{arXiv preprint arXiv:2204.05587}, 2022.

\bibitem[Herschlag et~al.(2020)Herschlag, Mattingly, Sachs, and
  Wyse]{herschlag2020non}
G.~Herschlag, J.~C. Mattingly, M.~Sachs, and E.~Wyse.
\newblock Non-reversible {M}arkov chain {M}onte {C}arlo for sampling of
  districting maps.
\newblock \emph{arXiv preprint arXiv:2008.07843}, 2020.

\bibitem[Hildebrand(1997)]{hildebrand1997rates}
M.~Hildebrand.
\newblock Rates of convergence for a non-reversible {M}arkov chain sampler.
\newblock \emph{preprint}, 1997.

\bibitem[Hsu et~al.(2019)Hsu, Kontorovich, Levin, Peres, Szepesv\'{a}ri, and
  Wolfer]{hsu2019}
D.~Hsu, A.~Kontorovich, D.~A. Levin, Y.~Peres, C.~Szepesv\'{a}ri, and
  G.~Wolfer.
\newblock Mixing time estimation in reversible {M}arkov chains from a single
  sample path.
\newblock \emph{Ann. Appl. Probab.}, 29\penalty0 (4):\penalty0 2439--2480, 08
  2019.
\newblock \doi{10.1214/18-AAP1457}.

\bibitem[Hsu et~al.(2015)Hsu, Kontorovich, and Szepesvari]{NIPS2015_7ce3284b}
D.~J. Hsu, A.~Kontorovich, and C.~Szepesvari.
\newblock Mixing time estimation in reversible {M}arkov chains from a single
  sample path.
\newblock In \emph{Advances in Neural Information Processing Systems},
  volume~28, 2015.

\bibitem[Kaniel(1966)]{kaniel1966estimates}
S.~Kaniel.
\newblock Estimates for some computational techniques in linear algebra.
\newblock \emph{Mathematics of Computation}, 20\penalty0 (95):\penalty0
  369--378, 1966.

\bibitem[Kotsalis(2022)]{kotsalis2022tractable}
G.~Kotsalis.
\newblock \emph{Tractable approximations and algorithmic aspects of
  optimization under uncertainty}.
\newblock PhD thesis, Georgia Institute of Technology, 2022.

\bibitem[Kuczy{\'n}ski and Wo{\'z}niakowski(1992)]{kuczynski1992estimating}
J.~Kuczy{\'n}ski and H.~Wo{\'z}niakowski.
\newblock Estimating the largest eigenvalue by the power and {L}anczos
  algorithms with a random start.
\newblock \emph{SIAM journal on matrix analysis and applications}, 13\penalty0
  (4):\penalty0 1094--1122, 1992.

\bibitem[Levin and Peres(2016)]{levin2016estimating}
D.~A. Levin and Y.~Peres.
\newblock Estimating the spectral gap of a reversible {M}arkov chain from a
  short trajectory.
\newblock \emph{arXiv preprint arXiv:1612.05330}, 2016.

\bibitem[Levin et~al.(2009)Levin, Peres, and Wilmer]{levin2009markov}
D.~A. Levin, Y.~Peres, and E.~L. Wilmer.
\newblock \emph{{M}arkov chains and mixing times, second edition}.
\newblock American Mathematical Soc., 2009.

\bibitem[Li et~al.(2023)Li, Lan, and Pananjady]{li2023accelerated}
T.~Li, G.~Lan, and A.~Pananjady.
\newblock Accelerated and instance-optimal policy evaluation with linear
  function approximation.
\newblock \emph{SIAM Journal on Mathematics of Data Science}, 5\penalty0
  (1):\penalty0 174--200, 2023.

\bibitem[Mcdonald et~al.(2011)Mcdonald, Shalizi, and
  Schervish]{mcdonald2011estimating}
D.~Mcdonald, C.~Shalizi, and M.~Schervish.
\newblock Estimating beta-mixing coefficients.
\newblock In \emph{Proceedings of the Fourteenth International Conference on
  Artificial Intelligence and Statistics}, pages 516--524, 2011.

\bibitem[Mohri and Rostamizadeh(2007)]{mohri2007stability}
M.~Mohri and A.~Rostamizadeh.
\newblock Stability bounds for non-iid processes.
\newblock \emph{Advances in Neural Information Processing Systems}, 20, 2007.

\bibitem[Montenegro et~al.(2006)Montenegro, Tetali,
  et~al.]{montenegro2006mathematical}
R.~Montenegro, P.~Tetali, et~al.
\newblock Mathematical aspects of mixing times in {M}arkov chains.
\newblock \emph{Foundations and Trends{\textregistered} in Theoretical Computer
  Science}, 1\penalty0 (3):\penalty0 237--354, 2006.

\bibitem[Neal(2004)]{neal2004improving}
R.~M. Neal.
\newblock Improving asymptotic variance of {MCMC} estimators: Non-reversible
  chains are better.
\newblock \emph{Technical Report No. 0406, Dept. of Statistics, University of
  Toronto}, 2004.

\bibitem[Ortner(2020)]{ortner2020regret}
R.~Ortner.
\newblock Regret bounds for reinforcement learning via {M}arkov chain
  concentration.
\newblock \emph{Journal of Artificial Intelligence Research}, 67:\penalty0
  115--128, 2020.

\bibitem[Paige(1971)]{paige1971computation}
C.~C. Paige.
\newblock \emph{The computation of eigenvalues and eigenvectors of very large
  sparse matrices.}
\newblock PhD thesis, University of London, 1971.

\bibitem[Paulin(2015)]{paulin2015concentration}
D.~Paulin.
\newblock Concentration inequalities for {M}arkov chains by {M}arton couplings
  and spectral methods.
\newblock \emph{Electronic Journal of Probability}, 20:\penalty0 1--32, 2015.

\bibitem[Paulsen(2002)]{paulsen2002completely}
V.~Paulsen.
\newblock \emph{Completely bounded maps and operator algebras}.
\newblock 78. Cambridge University Press, 2002.

\bibitem[Power and Goldman(2019)]{power2019accelerated}
S.~Power and J.~V. Goldman.
\newblock Accelerated sampling on discrete spaces with non-reversible {M}arkov
  processes.
\newblock \emph{arXiv preprint arXiv:1912.04681}, 2019.

\bibitem[Saad(1980)]{saad1980rates}
Y.~Saad.
\newblock On the rates of convergence of the lanczos and the block-lanczos
  methods.
\newblock \emph{SIAM Journal on Numerical Analysis}, 17\penalty0 (5):\penalty0
  687--706, 1980.

\bibitem[Shalizi and Kontorovich(2013)]{shalizi2013predictive}
C.~Shalizi and A.~Kontorovich.
\newblock Predictive {PAC} learning and process decompositions.
\newblock \emph{Advances in neural information processing systems}, 26, 2013.

\bibitem[Steinwart et~al.(2009)Steinwart, Hush, and
  Scovel]{steinwart2009learning}
I.~Steinwart, D.~Hush, and C.~Scovel.
\newblock Learning from dependent observations.
\newblock \emph{Journal of Multivariate Analysis}, 100\penalty0 (1):\penalty0
  175--194, 2009.

\bibitem[Stewart and Sun(1990)]{stewart1990matrix}
G.~W. Stewart and J.-g. Sun.
\newblock \emph{Matrix perturbation theory}.
\newblock Academic Press, 1990.

\bibitem[Sun et~al.(2010)Sun, Schmidhuber, and Gomez]{sun2010improving}
Y.~Sun, J.~Schmidhuber, and F.~J. Gomez.
\newblock Improving the asymptotic performance of {M}arkov chain
  {M}onte-{C}arlo by inserting vortices.
\newblock In \emph{Advances in Neural Information Processing Systems}, pages
  2235--2243, 2010.

\bibitem[Suwa and Todo(2010)]{suwa2010markov}
H.~Suwa and S.~Todo.
\newblock {M}arkov chain {M}onte {C}arlo method without detailed balance.
\newblock \emph{Physical review letters}, 105\penalty0 (12):\penalty0 120603,
  2010.

\bibitem[Syed et~al.(2022)Syed, Bouchard-C{\^o}t{\'e}, Deligiannidis, and
  Doucet]{syed2022non}
S.~Syed, A.~Bouchard-C{\^o}t{\'e}, G.~Deligiannidis, and A.~Doucet.
\newblock Non-reversible parallel tempering: A scalable highly parallel {MCMC}
  scheme.
\newblock \emph{Journal of the Royal Statistical Society Series B}, 84\penalty0
  (2):\penalty0 321--350, 2022.

\bibitem[Tropp(2015)]{tropp2015introduction}
J.~Tropp.
\newblock An introduction to matrix concentration inequalities.
\newblock \emph{Foundations and Trends{\textregistered} in Machine Learning},
  8\penalty0 (1-2):\penalty0 1--230, 2015.

\bibitem[Tropp(2012)]{tropp2012user}
J.~A. Tropp.
\newblock User-friendly tail bounds for sums of random matrices.
\newblock \emph{Foundations of computational mathematics}, 12\penalty0
  (4):\penalty0 389--434, 2012.

\bibitem[Truong(2022{\natexlab{a}})]{truong2022generalization}
L.~V. Truong.
\newblock Generalization error bounds on deep learning with {M}arkov datasets.
\newblock \emph{Advances in Neural Information Processing Systems},
  35:\penalty0 23452--23462, 2022{\natexlab{a}}.

\bibitem[Truong(2022{\natexlab{b}})]{truong2022kernel}
L.~V. Truong.
\newblock Generalization bounds on multi-kernel learning with mixed datasets.
\newblock \emph{arXiv preprint arXiv:2205.07313}, 2022{\natexlab{b}}.

\bibitem[Turitsyn et~al.(2011)Turitsyn, Chertkov, and
  Vucelja]{turitsyn2011irreversible}
K.~S. Turitsyn, M.~Chertkov, and M.~Vucelja.
\newblock Irreversible {M}onte {C}arlo algorithms for efficient sampling.
\newblock \emph{Physica D: Nonlinear Phenomena}, 240\penalty0 (4-5):\penalty0
  410--414, 2011.

\bibitem[Vucelja(2016)]{vucelja2016lifting}
M.~Vucelja.
\newblock Lifting-a nonreversible {M}arkov chain {M}onte {C}arlo algorithm.
\newblock \emph{American Journal of Physics}, 84\penalty0 (12):\penalty0
  958--968, 2016.

\bibitem[Wolfer(2020)]{pmlr-v117-wolfer20a}
G.~Wolfer.
\newblock Mixing time estimation in ergodic {M}arkov chains from a single
  trajectory with contraction methods.
\newblock In \emph{Proceedings of the 31st International Conference on
  Algorithmic Learning Theory}, volume 117 of \emph{Proceedings of Machine
  Learning Research}, pages 890--905. PMLR, 2020.

\bibitem[Wolfer(2022)]{wolfer2022empirical}
G.~Wolfer.
\newblock Empirical and instance-dependent estimation of {M}arkov chain and
  mixing time.
\newblock \emph{arXiv:1912.06845}, 2022.

\bibitem[Wolfer and Kontorovich(2019)]{pmlr-v99-wolfer19a}
G.~Wolfer and A.~Kontorovich.
\newblock Estimating the mixing time of ergodic {M}arkov chains.
\newblock In \emph{Proceedings of the Thirty-Second Conference on Learning
  Theory}, volume~99 of \emph{Proceedings of Machine Learning Research}, pages
  3120--3159, Phoenix, USA, 25--28 Jun 2019. PMLR.

\bibitem[Wolfer and Kontorovich(2021)]{wolfer2021}
G.~Wolfer and A.~Kontorovich.
\newblock Statistical estimation of ergodic {M}arkov chain kernel over discrete
  state space.
\newblock \emph{Bernoulli}, 27\penalty0 (1):\penalty0 532--553, 02 2021.
\newblock \doi{10.3150/20-BEJ1248}.

\bibitem[Wolfer and Watanabe(2021)]{wolfer2021information}
G.~Wolfer and S.~Watanabe.
\newblock Information geometry of reversible {M}arkov chains.
\newblock \emph{{Information Geometry}}, 4\penalty0 (2):\penalty0 393–--433,
  12 2021.
\newblock ISSN 2511-2481.

\bibitem[Yu(1994)]{yu1994rates}
B.~Yu.
\newblock Rates of convergence for empirical processes of stationary mixing
  sequences.
\newblock \emph{The Annals of Probability}, pages 94--116, 1994.

\bibitem[Zweig and Bruna(2020)]{zweig2020provably}
A.~Zweig and J.~Bruna.
\newblock Provably efficient third-person imitation from offline observation.
\newblock In \emph{Conference on Uncertainty in Artificial Intelligence}, pages
  1228--1237. PMLR, 2020.

\end{thebibliography}
\bibliographystyle{abbrvnat}

\end{document}